\documentclass[11pt, twoside, letter]{article}
\usepackage[margin = 1in]{geometry}
\usepackage{array}
\usepackage{subcaption}
\usepackage{amsthm}
\usepackage{amsmath}
\usepackage{amsfonts}
\usepackage{microtype}
\usepackage{graphicx}
\usepackage{booktabs} 
\newtheorem{theorem}{Theorem}[section]
\newtheorem{lemma}[theorem]{Lemma}
\newtheorem{proposition}[theorem]{Proposition}
\newtheorem{corollary}[theorem]{Corollary}
\newtheorem{assumption}[theorem]{Assumption}
\theoremstyle{definition}
\newtheorem{definition}[theorem]{Definition}

\usepackage{color}
\usepackage{graphicx}
\usepackage{pgfplots}
\usepackage{algorithm} 
\usepackage{algpseudocode}
\usepackage{mathrsfs}
\usepackage{caption}

\usepackage{xr}
\theoremstyle{remark}
\newtheorem{remark}[theorem]{Remark}

\newcommand{\diag}{\mathop{\mathbf{diag}}}
\newcommand{\abs}[1]{\ensuremath{\left|#1\right|}}
\newcommand{\norm}[2][]{\ensuremath{\left\Vert #2 \right\Vert{#1}}}

\renewcommand{\vec}[1]{\mathbf{#1}}

\newcommand{\grad}{\mathrm{grad}}

\newcommand{\Retr}{\mathrm{Retr}}
\newcommand{\M}{\mathcal{M}}
\newcommand{\Hess}{\mathrm{Hess}}

     \newcommand{\BN}{{\mathbb {N}}}
     
     \newcommand{\BR}{{\mathbb {R}}}

     \newcommand{\BZ}{{\mathbb {Z}}}

\begin{document}

\title{Constants of Motion: The Antidote to Chaos in Optimization and Game Dynamics}

\author{ Georgios Piliouras\\Singapore University of Technology and Design\\georgios@sutd.edu.sg
\and Xiao Wang\\Shanghai University of Finance and Economics\\wangxiao@sufe.edu.cn

}


\date{}
\maketitle

\begin{abstract}
Several recent works in online optimization and game dynamics have established strong negative complexity results including the formal emergence of instability and chaos even in small such settings, e.g., $2\times 2$ games. These results motivate the following question: Which methodological tools can guarantee the regularity of such dynamics and how can we apply them in standard settings of interest such as discrete-time first-order optimization dynamics? We show how proving the existence of invariant functions, i.e., constant of motions, is a fundamental contribution in this direction and establish a plethora of such positive results (e.g. gradient descent, multiplicative weights update, alternating gradient descent and manifold gradient descent) both in optimization as well as in game settings. At a technical level, for some conservation laws we provide an explicit and concise closed form, whereas for other ones we present non-constructive proofs using tools from dynamical systems.
\end{abstract}

\section{Introduction}
Optimization-driven learning dynamics lie at the core of some of the most successful ML applications be it Deep Learning~\cite{goodfellow2016deep}, where the goal is the minimization of a single loss function or 
Generative Adversarial Networks (GANs)~\cite{Goodfellow:2014:GAN:2969033.2969125}, where the goal is the simultaneous optimization over numerous utility functions, i.e. a game. Despite this high level affinity, the behavior of standard first-order optimization dynamics in the two settings are any but similar.

A common thread in the analysis of different  optimization algorithms (e.g. gradient descent) and heuristics in non-convex optimization  is that they have successfully leveraged deep intuition and tools coming from dynamical systems (e.g. stable manifold theorem) to understand the geometry of these orbits and argue that they converge not merely to fixed points but typically to local minima \cite{lee2016gradient,panageas2017gradient,du2017gradient,LPPSJR17,jin2017escape,PPW,PPW2019NIPS}. In some special cases of non-convex problems, e.g., when all local minima are global and there exists a negative curvature for every saddle point these results suffice to argue convergence to global optimum. However, these assumptions are quite restrictive since even for three-layer
linear networks, there exists a saddle point without a negative curvature~\cite{kawaguchi2016deep}. So, although positive results in the direction of global optimization exist (e.g., \cite{ge2016matrix,ge2017no,du2017gradient,du2019gradient}) it is clear that  we need a more fine grained language for understanding the behavior of such algorithms. Examples of hard gradient-like systems with multiple attractors of widely different quality are abound, e.g. in standard classes of potential/coordination games \cite{Kleinberg09multiplicativeupdates,anshelevich2008price}. If in the extreme we allow for time-varying online systems, then recent results have established formally  their worst case hardness \cite{Vaggos}. 

Whatever the difficulties raised in the non-convex optimization setting, these pale in comparison to the ones 
we are faced with in the study of optimization-driven learning in games, where instability and even chaos seems to be the norm \cite{galla2013complex,sanders2018prevalence}.  
In adversarial (zero-sum) games optimization dynamics such as Gradient Descent Ascent (GDA) or Multiplicative Weights Update (MWU)  do not converge to Nash equilibria but 
instead  can lead to  cycles \cite{mertikopoulos2018cycles,vlatakis2019poincare,piliouras2014optimization,balduzzi2018mechanics,bailey2020finite}, divergence\cite{BaileyEC18,cheung2018multiplicative}, or chaos \cite{SatoFarmer_PNAS,cheung2019vortices}. 
In fact, all Follow-the-Regularized-Leader dynamics, despite their optimal regret guarantees~\cite{Shalev12}, fail to achieve even local asymptotic stability on \textit{any equilibrium of any game} that does not admit a pure Nash equilibrium \cite{flokas2020no}!
In the face of such strongly negative instability and chaos results for standard optimization algorithms a lot of effort has been concentrated on the development of novel algorithms with provable guarantees in zero-sum games~\cite{daskalakis2018training,mertikopoulos2019optimistic,gidel2019a,mescheder2018training,yazici2018unusual}.
 However, such ad-hoc piecemeal approaches do not generalize well. 
 Indeed, recent work in the space of differential/smooth games \cite{balduzzi2018mechanics} has established the existence of such games 
where no reasonable gradient-based method converges \cite{letcher2021impossibility}!
 %
Perhaps even more alarming is the fact that even in the case of potential games~\cite{rosenthal73} (e.g. congestion or network coordination games) where the incentives of the agents are in perfect alignment, MWU dynamics can still bifurcate into instability and chaos especially in the presence of many agents \cite{palaiopanos2017multiplicative,CFMP,Thip18}. 

Naturally, a complete understanding and classification of all possible behaviors of these optimization driven systems  is arguably too aggressive a goal. However, in practice, 
 it would be helpful for us to know a sufficient condition under which the dynamical system is not chaotic, or at least has some notion of structure and predictability. Thus, we are driven by the following question: 

\emph{In the case of MWU, or more generally, first-order optimization algorithms in higher dimensions, what makes the induced dynamical systems non-chaotic?}

\paragraph{Our results and contributions.} 
We answer the above question by showing the following:
\begin{enumerate}
\item For Alternating Gradient Descent in coordination games, the existence of invariant functions negates the occurrence of chaos into a zero-measure set, Proposition \ref{levelset} and Theorem \ref{thm:bipameasure0};  

\item For Gradient Descent, MWU and Manifold Gradient Descent with small step size, there are many invariant functions on a open dense subset of the phase space.

\end{enumerate}


\subsection{Related Work}
{\bf Invariants in game and optimization dynamics} The question of existence of invariant function have been the subject of recent work in the area of adversarial machine learning.
 Zero-sum game dynamics are typically cyclic or slowly divergent and even when they do converge they typically do so in a spiralling fashion \cite{daskalakis2018training,daskalakis2018last,gidel2019a,yazici2018unusual,gidel2019negative,balduzzi2019open,BaileyEC18,2019arXiv190602027A}. 
 At the core of these results lie formal connections between such dynamics and Hamilotnian systems, i.e., systems that have a notion of ``constant of the motion" (Hamiltonian) that weaves the dynamics into recurrent periodic-like orbits \cite{BP,2018arXiv180205642B,2019arXiv190602027A,vlatakis2019poincare}. Arguably, in the most closely related paper to ours, 
 \cite{nagarajan2020chaos} recently showed how to identify enough invariant functions in continuous-time variants of Follow-the-Regularized-Leader dynamics and then used dimensionality reduction arguments  to argue that chaos was not possible in a class of network polymatrix games. Critically, their technique was based on the Poincar\'{e}-Bendixson theorem that states that ODEs with two degrees of freedom cannot exhibit chaos. However, this theorem/proof technique does not apply in discrete, where even one dimensional (game theoretic) maps can be chaotic~\cite{palaiopanos2017multiplicative}. In fact, making progress in the case of discrete-time dynamics, that we exactly address here, was the main open question of that paper.
 
 We build upon  early insights in \cite{Stebe} where invariants  were established for an update rule known as multiplicative weights update  (MWU) \cite{Arora05themultiplicative} when the potential function is a polynomial with non-negative coefficients defined over a simplex. Our analysis generalizes greatly both the class of dynamics as well as the class of functions for which invariant functions  exist.

 \section{Preliminaries}
\textbf{Notations}  Throughout this paper, we use bold font $\vec{x}$ for vectors and regular $x$ for points in manifold, $\mathbb{Z}$ denotes the set of integers, $T$ denote the transformation defined by optimization algorithm ($T_{\eta}$ is used if stepsize $\eta$ is specified), $[\vec{x}]$ denotes the orbit generated by iteration of $T$ where $\vec{x}$ belongs, i.e. $[\vec{x}]:=\{T^k(\vec{x})\}_{k\in\mathbb{Z}}$. The $\nabla f$ and $\grad f$ refer to the Euclidean gradient and general Riemannian gradient, $\nabla^2f$ and $\Hess f$ refer to the Euclidean and Riemannian Hessian.
\\
A multi-index is an $n$-tuple of nonnegative integers, denoted as
$
\alpha=(\alpha_1,\alpha_2,...,\alpha_n) 
$
where $\alpha_j\in\{0,1,2,...\}$.
If $\alpha$ is a multi-index, we define
$
\abs{\alpha}=\alpha_1+\alpha_2+...+\alpha_n, \ \ \ \alpha!=\alpha_1!\alpha_2!...\alpha_n!
$,
$
\vec{x}^{\alpha}=x_1^{\alpha_1}x_2^{\alpha_2}...x_n^{\alpha_n}, \text{where} \ \ \vec{x}=(x_1,...,x_n)\in\mathbb{R}^n
$,
$
\partial^{\alpha}f=\partial_1^{\alpha_1}\partial_2^{\alpha_2}...\partial_n^{\alpha_n}f=\frac{\partial^{|\alpha|}f}{\partial x_1^{\alpha_1}\partial x_2^{\alpha_2}...\partial x_n^{\alpha_n}}
$.
\paragraph{Dynamical System}
Let $T:\mathbb{R}^n\rightarrow\mathbb{R}^n$ be a differentiable map. The process $\vec{x}_{k+1}=T(\vec{x}_k)$ for $k\in\mathbb{Z}$ is called a (discrete time) dynamical system in $\mathbb{R}^n$.

\paragraph{Homeomorphism and Diffeomorphism}
A map $T:X\rightarrow Y$ between two topological spaces is a homeomorphism if $T$ is a continuous  bijection such that the inverse $T^{-1}$ is also continuous. $T$ is a diffeomorphism if it is a homeomorphism such that both $T$ and $T^{-1}$ are differentiable.

\begin{theorem}[Taylor's Theorem]\label{Taylor}
Suppose $f:\mathbb{R}^n\rightarrow\mathbb{R}$ is of class $C^{k+1}$ on an open convex set $S$. If $\vec{a}\in S$ and $\vec{a+h}\in S$, then
\[
f(\vec{a}+\vec{h})=\sum_{|\alpha|\le k}\frac{\partial^{\alpha}f(\vec{a})}{\alpha!}\vec{h}^{\alpha}+R_{\vec{a},k}(\vec{h}),
\]
where the remainder is given in Lagrange's form by
\[
R_{\vec{a},k}(\vec{h})=\sum_{|\alpha|=k+1}\partial^{\alpha}f(\vec{a}+c\vec{h})\frac{\vec{h}^{\alpha}}{\alpha!} \ \ \ \text{for some} c\in(0,1).
\]
\end{theorem}

\begin{corollary}
If $f$ is of class $C^{k+1}$ on $S$ and $\abs{\partial^{\alpha}f(\vec{x})}\le L$ for $\vec{x}\in S$ and $\abs{\alpha}=k+1$, then
\[
\abs{R_{\vec{a},k}(\vec{h})}\le \frac{L}{(k+1)!}\norm{\vec{h}}^{k+1}
\]
where
\[
\norm{\vec{h}}=\abs{h_1}+\abs{h_2}+...+\abs{h_n}.
\]
\end{corollary}

\paragraph{Proper Maps}
Let $X$ and $Y$ be topological spaces. A map from $X$ to $Y$, denoted $f:X\rightarrow Y$, is called \emph{proper} if the inverse of each compact subset of $Y$ is a compact subset of $X$.

\begin{theorem}[Theorem 2, \cite{Ho1975}]\label{proper map}
Let $X$ be path-connected and $Y$ be simply-connected Hausdorff spaces. A local homeomorphism $f:X\rightarrow Y$ is a global homeomorphism of $X$ to $Y$ if and only if the map $f$ is proper.
\end{theorem}

Dynamical systems that can be fully understood are said to be regular and some other dynamical systems have strange, \emph{chaotic} behavior \cite{palaiopanos2017multiplicative}. There are many different definitions for the concept of being ``chaotic". In this paper we follow the definition of Li-Yorke chaos that is mostly used in machine learning community in recent years, e.g. \cite{cheung2019vortices}. In the rest of the first section, we review some fundamental concepts in the theory of chaos.

\paragraph{Scrambled Set}
Let $T:\BR^n\rightarrow\BR^n$ be a dynamical system with $T$ the update rule. A pair of points $\vec{x},\vec{y}\in\BR^d$ is called \emph{scrambled} if $\liminf_{k\rightarrow\infty}\norm{T^k(\vec{x})-T^k(\vec{y})}=0$ and also $\limsup_{k\rightarrow\infty}\norm{T^k(\vec{x})-T^k(\vec{y})}>0$. A set $S$ is called ``scrambled" if for all $\vec{x},\vec{y}\in S$, the pair is ``scrambled".

\begin{remark}
The term ``Chaos" was introduced by \cite{liyorke} for the first time in describing a complicated and highly irregular behavior of one dimenisonal dynamical system on an interval of $\BR$.
\end{remark} The formal statement is the following:
\paragraph{Li-Yorke Chaos}\label{LYchaos}
A discrete time dynamical system with update rule $f:I\rightarrow I$ is called chaotic if (a) for each $k\in\BZ^+$, there exists a periodic point $p\in I$ of period $k$ and (b) there is an uncountably infinite set $S\subset I$ that is ``scrambled".

\begin{remark}
A deep result of \cite{liyorke} and \cite{sharkovskii} asserts that if a continuous map on a closed interval has periodic point of period 3, then this map has periodic point of each period $k\in\BZ^+$. Addition to \cite{palaiopanos2017multiplicative}, machine learning community has found more application of this theory in deep neural networks, e.g. \cite{CNPW2020,CNP2020}.
\end{remark}

\section{Usefulness of Invariant Functions}
This section reviews classic perspective of Invariant functions and Chaos, and goes through the recent development of Chaos in Machine Learning and Game Theory. Regarding the Non-existence of Chaos, Classification of Orbits and Dimensionality Reduction, the usefulness of invariant functions is illustrated in the following: 
\begin{figure}[H]
\centering
\includegraphics[width=0.5\textwidth]{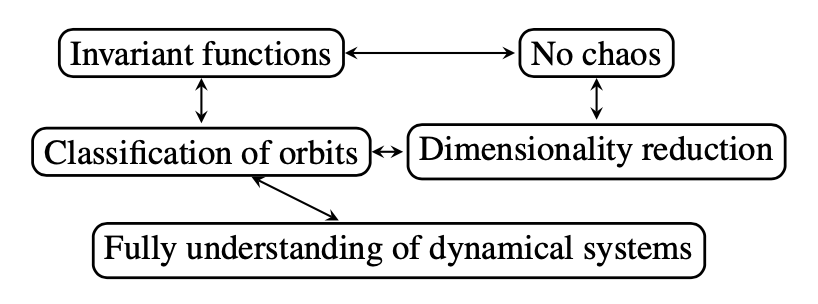}
\caption{Understanding dynamical systems based on invariant functions}
\label{usefulness}
\end{figure}

\subsection{Classic Perspective}
\paragraph{Background of Physics} A fundamental analysis of the order or the regularity of motion starts by representing our object of interest by a \emph{vector}, i.e. a point in the phase space. Generally, the search for order is not easy due to the complexity of the real world. For example, in physics the motion of an object, when described by an orbit, exhibits unlimited diversity and complexity. Under the superficial diversity of various \emph{orbits}, there is a deep structure that rules the generation of every motion: the \emph{equation of motion} in Newton mechanics. Under this framework, the natural method to uncover the order of motion is to find the \emph{integrals} of motion such that the motion is represented by some ``function". However, it is known that integrable equations are rather special, an equation of motion is generally non-integrable. In the theory of mechanics, the non-integrability is considered as chaos or irregularity. According to, \cite{zensho}, \emph{chaos} is defined as a motion that cannot be represented by a ``function".
\paragraph{Invariant Functions and Dimensionality Reduction} 
The order of motion is revealed by decomposing the parameters in the system and describing the change of each parameter by a ``function" that is a representation of an order. Take the linear system as an example, this decomposition is nothing but the eigenvalue problem. In the nonlinear regime, the structure of dynamics can be spanned by \emph{constants of motion},\cite{zensho}, i.e. certain quantity that does not change with time (discrete or continuous). Geometrically, a smooth curve in the space $\mathbb{R}^n$ may be viewed as an intersection of $n-1$ hypersurfaces. A hypersurface is represented by some equation $\Phi(\vec{x})=c$, where $\Phi:\mathbb{R}^n\rightarrow \BR$. Then an orbit is given as a set of points satisfying $n-1$ equations, 
\[
\Phi_j(\vec{x})=c_j \ \ \ \ \ (j=1,...,n-1).
\]
\begin{figure}[H]
\centering
\includegraphics[width=0.45\textwidth]{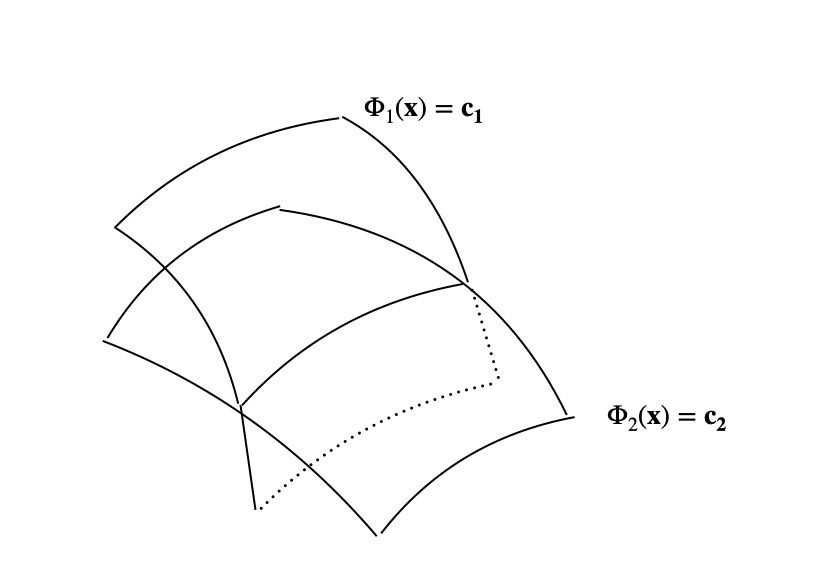}
\caption{Orbit as intersection of hypersurfaces}
\label{intersection}
\end{figure}
A system that is decomposable into constants of motion is equivalent to that the system is integrable. Note that if the orbit of a motion, denoted as $\vec{x}(t)$, is on the intersection of the above equations, then all of $\Phi_j(\vec{x})$ are \textbf{invariant} along the orbit, i.e. $\Phi_j(\vec{x}(t))=c_j$ for all $j\in[n]$. We suggest referring to Chapter 2 of \cite{zensho} for a detailed argument. 

With the intuition of Figure \ref{intersection}, we expect to understand the orbit as the intersection of the hypersurfaces of invariant function. If there exists one invariant function, then two points $\vec{x}_1$ and $\vec{x}_2$ must not be on the same orbit provided $\Phi(\vec{x}_1)\ne\Phi(\vec{x}_2)$. Usually the more invariant functions we can find for a dynamical system, the more efficiently we can reduce the dimension of the phase space into union of lower dimensional spaces.

\subsection{Extensions}

\paragraph{Chaos in Learning and Games}
The first example of chaos in machine learning and game dynamics is introduced by \cite{palaiopanos2017multiplicative}, where the authors investigate the (non)convergence of multiplicative weights update (MWU) in optimization of congestion games. Actually the meaning of ``chaos" in \cite{palaiopanos2017multiplicative} (Li-Yorke Chaos) is not the same as aforementioned one used in physics (non-integrability and non-decomposability). The main discovery of their paper is the following: For a congestion game, there exists certain step size for exponential MWU, such that the orbit might converge to limit cycle, or exhibits highly irregular behavior. Note that non-integrability is defined for continuous time dynamical systems, and this make it difficult to compare with Li-Yorke Chaos where the dynamical system is discrete. Despite of the difference between non-integrability and Li-Yorke Chaos, they can still be connected if we consider \emph{chaos} as ``\textbf{lacking of invariant functions}". We leave a detailed discussion on this in Section \ref{IF for CG}.

\paragraph{Our Perspective}
The high level idea of this paper is the following: The dichotomy of order and chaos of a dynamical system is related to invariant functions of the system, and roughly summarized by the following figure.

\begin{figure}[H]
\centering
\includegraphics[width=0.5\textwidth]{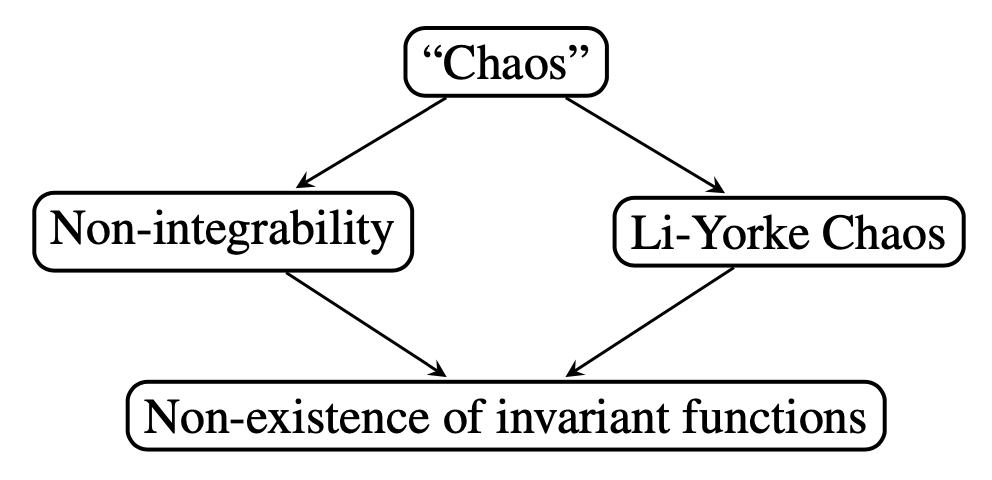}
\caption{Chaos and invariant functions}
\end{figure}

In Section \ref{IF for CG}, we show that for coordination game and bipartite network game, if the agents use Alternating Gradient Ascent, one can find at least one invariant function that is continuous everywhere. An immediate consequence is that the scrambled set generalized from the definition of Li-Yorke Chaos is lying in a zero measure set, which implies that ``chaos" is almost unlikely to occur. From the perspective of \emph{constant of motion}, we have reduced degree of freedom of coordination game by 1 since the points of the same orbit must be on the same level set of the invariant function $\Phi$. 

Section \ref{invariant functions for GD} and \ref{IF:MWU} show that we can actually be more optimistic in the decomposition of the dynamical system induced by gradient descent and multiplicative weights update if the step size is small, where we show that on a $d$-dimensional space $\mathbb{R}^d$, there exist $d$ independent invariant functions. However, these invariant functions are proven to be continuous on an \emph{open dense subset} of $\BR^d$ instead of the whole space. In the light of constants of motion, with $(d-1)$ invariant function, we can already determine that the orbit is on the intersection of these $(d-1)$ equations such that $\Phi_j(\vec{x})=c_j$, so we conclude that the optimization dynamics is highly regular at least on a \textbf{topologically large} subset, see Theorem \ref{invariant functions:GD} and Theorem \ref{IF:compact}. In fact we have obtained an equivalent description of the whole dynamical system: Denote $\phi_1(\vec{x}),...,\phi_d(\vec{x})$ are the invariant functions, then a pair of points $\vec{x}$ and $\vec{y}$ are on the same orbit if and only if the values of $\phi_i$ agree on $\vec{x}$ and $\vec{y}$ for all $i\in[d]$. Moreover, this is a parametrization of the ``space of orbits", since each set of real numbers $(c_1,...,c_d)$ satisfying $\phi_i(\vec{x})=c_i$ refers to a unique orbit $\{T^k(\vec{x})\}_{k\in\BN}$. 

\section{Negation of Chaos: Invariant Functions in Coordination Game}\label{IF for CG}
In this section we show that for the case of bipartite network games, one can find at least one invariant function with closed form. Furthermore, the existence of such invariant function, although we have just found one, implies that Li-Yorke like chaos can only occur in a measure zero set. 

\subsection{Invariant functions in Coordination Games with Alternating Play}
A bipartite network coordination game consists two groups of players $\mathcal{N}_1=\{1,...,n\}$ and $\mathcal{N}_2=\{1,...,m\}$. Suppose the payoff matrices between $i\in\mathcal{N}_1$ and $j\in\mathcal{N}_2$ is $A^{ij}$, and the players in $\mathcal{N}_1$ and $\mathcal{N}_2$ use $\eta_1$ and $\eta_2$ as the learning rates respectively. With Alternating play, the two groups of players update their strategies according to follows
\begin{align}\label{Bipalgorithm}
\vec{x}_i^{t+1}&=\vec{x}_i^t+\eta_1\sum_{j=1}^mA^{ij}\vec{y}_j^t, 
\\
\vec{y}_j^{t+1}&=\vec{y}_j^t+\eta_2\sum_{i=1}^nA^{ij\top}\vec{x}_i^{t+1}
\end{align}
where $\vec{x}_i$ is the strategy of player $i$ and $\vec{y}_j$ is the strategy of player $j$. Then we have a closed form of invariant function for Bipartite Network Game with Alternating Play from the following proposition. 
\begin{theorem}\label{bipartitegame}
Suppose in a bipartite network game, two groups of players update their strategy using alternating gradient descent. Then the function
\[
\Phi(X,Y)=\frac{1}{\eta_1}\sum_{i=1}^n\norm{\vec{x}_i}^2-\frac{1}{\eta_2}\sum_{j=1}^m\norm{\vec{y}_j}^2+\sum_{i=1}^n\sum_{j=1}^m\langle\vec{x}_i,A^{ij}\vec{y}_j\rangle
\]
is invariant under the update Alternating gradient descent, where $X=(\vec{x}_1,...,\vec{x}_n)$ and $Y=(\vec{y}_1,...,\vec{y}_m)$.
\end{theorem}
A two-agent (bilinear unconstrained) coordination game consists of two agents $\mathcal{N}=\{1,2\}$ where agent $i$ selects a strategy from $\mathbb{R}^{k_i}$. Utilities of both agents are determined via a payoff matrix $A\in\mathbb{R}^{k_1\times k_2}$. In a coordination game, both agents have utility $\langle \vec{x}_1,A\vec{x}_2\rangle$ provided agent 1 selects strategy $\vec{x}_1$ and agent 2 selects strategy $\vec{x}_2$. As a special case of bipartite game, the invariant function of two-agent alternating play is just $\Phi(\vec{x},\vec{y})=\frac{1}{\eta_1}\norm{\vec{x}}^2-\frac{1}{\eta_2}\norm{\vec{y}}^2+\langle\vec{x},A\vec{y}\rangle$.

\subsection{Scrambled Set of Measure Zero}
We next show our first main result that the existence of continuous invariant function implies the non-existence of "large" scrambled set.

\begin{proposition}\label{levelset}
Let $M$ be a complete metric space and $T:M\rightarrow M$ be a homeomorphism. Suppose that there exists a continuous function $\Phi(x):M\rightarrow\mathbb{R}$ such that $\Phi(T(x))=\Phi(x)$ for all $x\in M$. Then the maximal scrambled set $S$ is constrained in a level set of the invariant function $\Phi(x)$, i.e. $S\subset \{x:\Phi(x)=c\}$ for some $c\in\BR$.
\end{proposition}

\begin{proof}
We prove this by contradiction. Assume that there exists a pair of points $x_1$ and $x_2$ in the scrambled set, such that $\Phi(x_1)=c_1$, $\Phi(x_2)=c_2$ and $c_1\ne c_2$. By the definition of scrambled set, we have that 
\[
\liminf_kd(T^k(x_1)-T^k(x_2))=0
\]
and this implies that there exists a subsequence of integers $\{k_i\}_{i\in\BZ}$ such that
\[
\lim_{i\rightarrow\infty}d(T^{k_i}(x_1)-T^{k_i}(x_2))=0.
\]
Since $\Phi(x)$ is continuous in the whole space $M$, for the sequences $\{T^{k_i}(x_1)\}_i$ and $\{T^{k_i}(x_2)\}_i$, we have 
\begin{equation}\label{eq:lim0}
\lim_{i\rightarrow\infty}\abs{\Phi(T^{k_i}(x_1))-\Phi(T^{k_i}(x_2))}=0.
\end{equation}
However, by the invariance of $\Phi$ under the iteration of $T$, we have $\Phi(T^{k_i}(x_1))=\Phi(x_1)=c_1$ and $\Phi(T^{k_i}(x_2))=\Phi(x_2)=c_2$, and then 
\begin{align}
&\abs{\Phi(T^{k_i}(x_1))-\Phi(T^{k_i}(x_2))}
\\
&=\abs{\Phi(x_1)-\Phi(x_2)}
\\
&=\abs{c_1-c_2}>0,
\end{align}
contradicting to equation (\ref{eq:lim0}). The proof completes.
\end{proof}

\begin{theorem}\label{thm:bipameasure0}
For the coordination game and bipartite network game, suppose that the players update their strategies with Alternating Gradient Ascent, i.e., algorithm (\ref{Bipalgorithm}). Then the scrambled set has Lebesgue measure zero in the strategy space.
\end{theorem}

\begin{proof}
By the Proposition \ref{levelset}, we know that the scrambled set is contained in a level set of the invariant function $\Phi$. Since $\Phi$ is a quadratic function, the differential $D\Phi$ can be written as a matrix multiplied by the vector $(X,Y)$. On the set where $\Phi(X,Y)$ is a constant, the differential $D\Phi$ is identically $\vec{0}$. By looking at the differential of $\Phi$ in the two player coordination game, we compute the differential $D\Phi=H\vec{x}$ where
\[
H=\left[
\begin{array}{cc}
\diag\{\frac{2}{\eta_1}\}&A
\\
A&\diag\{\frac{2}{\eta_2}\}
\end{array}
\right].
\]
The rank of $H$ is at least 2 since it contains a diagonal matrix $\diag\{\frac{2}{\eta_1}\}$. Therefore the set of points belongs to a lower dimensional vector space. The same argument extends to the bipartite network game since the differential $D\Phi=H'(X,Y)^{\top}$ where $H'$ contains independent row vectors. Combining with Proposition \ref{levelset}, we conclude that the level set $\{(X,Y):\Phi(X,Y)=c\}$ has measure zero.
\end{proof}

\subsection{Simulations}
 We illustrate the result of Theorem \ref{bipartitegame} by running the Alternating Gradient Ascent on the function $f(x,y)=xy$. This is a special case of the utility function $f(\vec{x},\vec{y})=\langle \vec{x},A\vec{y}\rangle$ when $\vec{x},\vec{y}$ and $A$ are all scalar. The algorithm is then written as
\begin{align}
x^{t+1}&=x^t+\eta_1y^t
\\
y^{t+1}&=y^t+\eta_2x^{t+1}.
\end{align}
Therefore, the invariant function is 
\[
\Phi(x,y)=\frac{x^2}{\eta_1}-\frac{y^2}{\eta_2}+xy,
\]
whose graph is a \textbf{hyperboloid} in Figure \ref{orbitlevel1} and \ref{orbitlevel2}.
The simulation verifies the following two facts:
\begin{enumerate}
\item Each trajectory lies on a unique level curve of the function $\Phi(x,y)$, i.e. the dynamical system is regular;
\item Regardless of whether the learning rates are small or large, the regularity always holds. 
\end{enumerate}

\begin{figure}[h]
\centering
\includegraphics[width=0.35\textwidth]{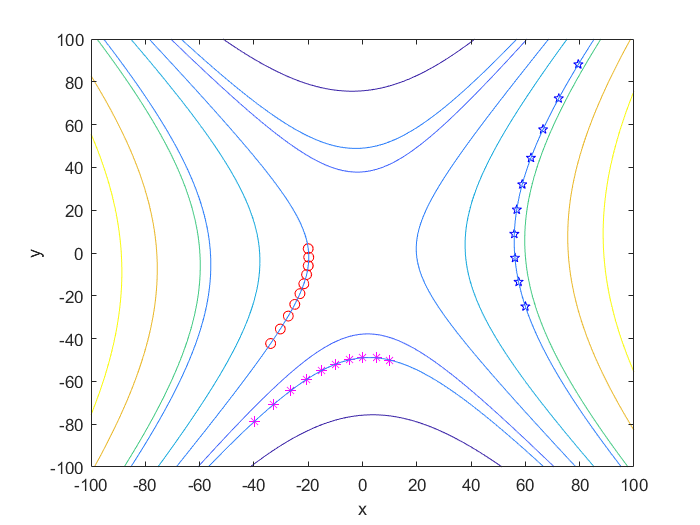}
\hspace{1.2cm}
\includegraphics[width=0.35\textwidth]{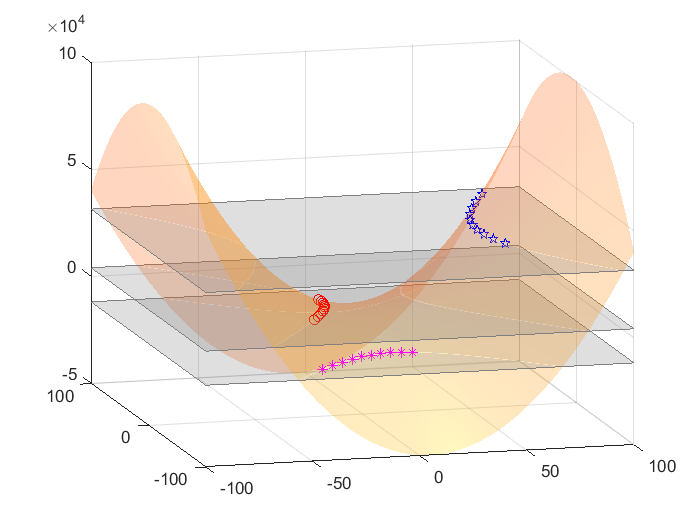}
\caption{Initial conditions $(x_0,y_0)=(60,-25)$, $(-20,2)$ and $(10,-50)$, the corresponding level sets (in descending order) are $\Phi(x,y)=31375$, $\Phi(x,y)=3940$ and $\Phi(x,y)=-12000$; learning rate $(\eta_1,\eta_2)=(0.1,0.2)$.}
\label{orbitlevel1}
\end{figure}

\begin{figure}[h]
\centering
\includegraphics[width=0.35\textwidth]{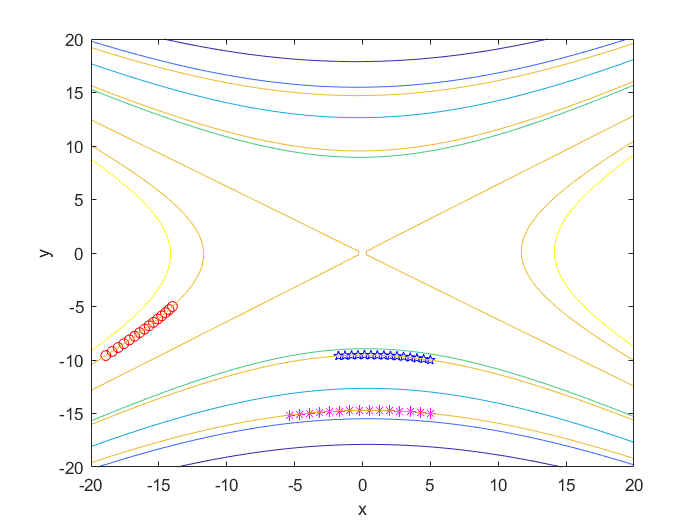}
\hspace{1.2cm}
\includegraphics[width=0.35\textwidth]{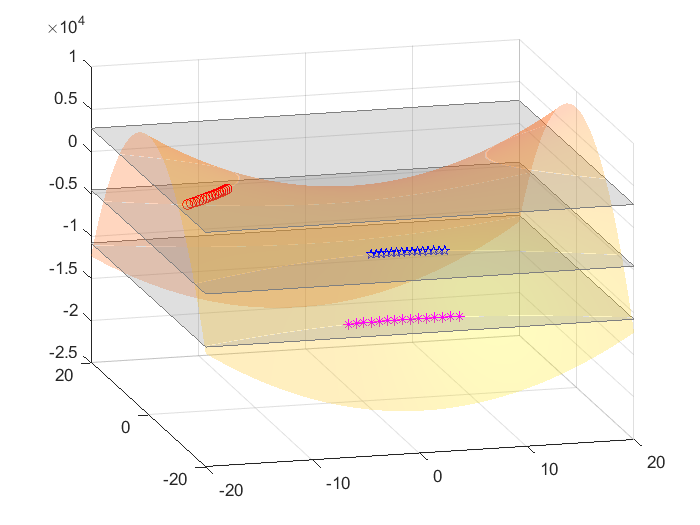}
\caption{Initial conditions $(x_0,y_0)=(-14,-5)$, $(5,-10)$ and $(5,-15)$, the corresponding level sets (in descending order) are $\Phi(x,y)=2740$, $\Phi(x,y)=-4550$ and $\Phi(x,y)=-10825$; learning rate $(\eta_1,\eta_2)=(0.05,0.02)$.}
\label{orbitlevel2}
\end{figure}

\section{Invariant Functions for First-order Methods}
So far we have found the closed form of invariant functions of alternating gradient descent in coordination games. In the rest of the paper, we explore the general theory of invariant functions for gradient descent and multiplicative weights update, both are arguably the most popular algorithms in multi-agent game theory and optimization. 

\subsection{Invariant functions for Gradient Descent}\label{invariant functions for GD}
In this section, we establish the existence and representation of invariant functions for gradient descent.

\begin{assumption}\label{assumption1}
Let $f\in C^2$ and $\abs{\frac{\partial^2f}{\partial x_i\partial x_j}(\vec{x})}\le L$ for all $\vec{x}\in\mathbb{R}^d$.
\end{assumption}
This Lipschitz-type assumption guarantees the gradient descent decreases the value of $f$ if the step-size is small.
\begin{assumption}\label{assumption2}
The gradient descent algorithm $T(\vec{x}):\mathbb{R}^d\rightarrow\mathbb{R}^d$ is a proper map.
\end{assumption}

\begin{remark}
This assumption guarantees that the optimization algorithm is a global homeomorphism if it is locally a diffeomorphism. By taking small enough step-size, one can only ensure that the Jacobian matrix of the gradient descent $T(\vec{x})=\vec{x}-\eta\nabla f(\vec{x})$, say $DT(\vec{x})$, is invertible everywhere. However, a local diffeomorphism is not necessarily globally invertible, unless the underlying space is simply-connected (each circle can be contracted to a point without being broken or blocked) and the map is proper. See \cite{Ho1975} for more details. 
\end{remark}

\begin{theorem}\label{invariant functions:GD}
Let $T(\vec{x})=\vec{x}-\eta\nabla f(\vec{x})$ be the gradient descent for the function $f:\mathbb{R}^d\rightarrow\mathbb{R}$ satisfying Assumptions \ref{assumption1} and \ref{assumption2}, and stepsize $\eta<\frac{2}{dL}$. Then $T$ is global diffeomorphism and there exist $d$ $T$-invariant functions, $\varphi_i(\vec{x}):\mathbb{R}^d\rightarrow\mathbb{R}, i\in[d]$, which are continuous and independent on an open dense subset of $\mathbb{R}^d-F$, where $F$ is the set of fixed point of $T$.
\end{theorem}
\begin{remark}
The \emph{``independent"} means that the set of $d$ functions, i.e., $(\varphi_1,...,\varphi_d):\mathbb{R}^d\rightarrow\mathbb{R}^d$ consist a local diffeomorphism.
\end{remark}

Actually the above Theorem \ref{invariant functions:GD} shows that there are ``many" continuous invariant functions for the dynamical system induced by iteration of gradient descent on a \textbf{topologically large} region (an open dense subset) of $\BR^d$. In contrast to the results for coordination game with alternating play, in which case we find one specific invariant function that is continuous on the whole space $\BR^d$, we cannot conclude that the scramble set is necessarily of measure zero, since we have no information in the complement of this open dense subset mentioned in Theorem \ref{invariant functions:GD}. 

Another difficulty in understanding the connection between optimization and chaos using Theorem \ref{invariant functions:GD} is that the theorem provides a purely existence guarantee, and this make it hard to use the result in application. Then we raise the following question naturally:

\emph{Even if we cannot always find a closed form of invariant function, can we find a way to approximate it? Moreover, can this approximation (if it exists) give insight to the connection with chaos?}

The answer to the first part of the question is affirmative. On the connection with chaos, we will show that a similar result on the invariant function and non-existance of scrambled set is guaranteed if the function $f$ satisfies additional assumptions on the topology of critical points.

\begin{theorem}\label{IF:construction}
Let $f\in C^2(\mathbb{R}^d)$ and $T$ be the homeomorphism defined by a gradient descent algorithm for $f$ on $\mathbb{R}^d$. Then for any continuous function $p(\vec{x}):\mathbb{R}^d\rightarrow\mathbb{R}$, the following infinite sum,
\[
\Phi(\vec{x})=\sum_{n=-\infty}^{\infty}p(T^n(\vec{x}))(f(T^{n-1}(\vec{x}))-f(T^n(\vec{x})))
\]
is a continous $T$-invariant function in a open dense set $G$ of $\mathbb{R}^d-F$. 
\end{theorem}

\begin{remark}
If the series is divergent, then $\Phi(\vec{x})$ is trivially $T$-invariant.
\end{remark}

\begin{corollary}\label{IF:construction discrete crit}
Suppose that $f\in C^2(\BR^d)$ has at most countably many critical points, denoted as $F$, and all the saddle points are strict in the sense that the minimal eigenvalue of the Hessian at each saddle point is strictly less than 0. Then for any bounded continuous function $p(\vec{x}):\BR^d\rightarrow \BR$, the infinite sum 
\[
\Phi(\vec{x})=\sum_{n=-\infty}^{\infty}p(T^n(\vec{x}))(f(T^{n-1}(\vec{x}))-f(T^n(\vec{x})))
\]
is continuous on an open dense subset $G$ of $\BR^d$ and the Lebesgue measure of the complement of $G$ in $\BR^d$ equals to zero. Furthermore, if all the level sets of $\Phi$ are measure zero for some continuous function $p$, then the scrambled set of the gradient descent is of measure zero.
\end{corollary}
This corollary shows that gradient descent is not chaotic (in the sense of Li-Yorke) almost always.

\subsection{Invariant Functions for Multiplicative Weights Update}\label{IF:MWU}
In this section, we focus on multiplicative weights update (MWU) and the invariant functions of the dynamical system induced by MWU. Formally, MWU and its linear variant $\text{MWU}_{\ell}$ are stated as follows
\begin{equation}\label{MWUexp}
x_{ij}\leftarrow \frac{x_{ij}\exp\left(-\epsilon_i\frac{\partial f}{\partial x_{ij}}(\vec{x})\right)}{\sum_sx_{is}\exp\left(-\epsilon_s\frac{\partial f}{\partial x_{is}}(\vec{x})\right)}
\end{equation}
and 
\begin{equation}\label{MWU}
x_{ij}\leftarrow x_{ij}\frac{1-\epsilon_i\frac{\partial f}{\partial x_{ij}}(\vec{x})}{1-\epsilon_i\sum_sx_{is}\frac{\partial f}{\partial x_{is}}(\vec{x})},
\end{equation}
where $f$ is a differentiable function constrained on $M=\{(x_{ij}):x_{ij}\ge0,\sum_{j=1}^{n_i}x_{ij}=1\ \ \text{for all}\ \ 1\le i\le N\}$, and $\epsilon_i$ is the learning rate of the agent $i$.

\begin{theorem}\label{IF:compact}
Let $M=\{(x_{ij}):x_{ij}\ge0,\sum_{j=1}^{n_i}x_{ij}=1\ \ \text{for all}\ \ 1\le i\le N\}$, Let $f:M\rightarrow\mathbb{R}$ be a differentiable function. Let $T$ be the MWU algorithm defined by (\ref{MWU}) or (\ref{MWUexp}). Then for small enough learning rates $\epsilon_i$, $i\in[N]$,  there exist $d=N(n_i-1)$ invariant functions which are continuous and independent on an open dense subset of $M-F$. Moreover, for any continuous function $p$ on $M$, the series
\[
\sum_{-\infty}^{\infty}p(T^n(\vec{x}))(f(T^{n-1}(\vec{x}))-f(T^{n}(\vec{x})))
\]
represents an invariant function continuous on $M-F$ where $F$ is the set of fixed points of $T$.
\end{theorem}  
\begin{remark}
The proof of Theorem \ref{IF:compact} relies on the compactness of $M$, in contrast to the proof of Theorem \ref{IF:construction} that introduces a regularity condition for $f$.  
\end{remark}

\subsection{Invariant Functions for Manifold Gradient Descent}
We generalize the existence of invariant functions to manifold gradient descent. With the former treatment of gradient descent and multiplicative weights update, we can accomplish this generalization with basic concepts of Riemannian manifold. We start with fundamentals on optimization on Riemannian manifold, and we recommend \cite{AMT} as a compact reference for the basics of Riemannian manifold optimization.
\paragraph{Riemannian Metric} Let $M$ be a $d$-dimensional manifold. At each point $\vec{x}\in M$ there associates a $d$-dimensional real vector space $T_{\vec{x}}M$, called the tangent space at $\vec{x}$. The \emph{Riemannian metric} is an inner product $\langle,\rangle_{\vec{x}}$ on the tangent space, with norm $\norm{\cdot}_{\vec{x}}=\sqrt{\langle,\rangle_{\vec{x}}}$.
\paragraph{Riemannian Gradient} The \emph{Riemannian gradient} $\grad f(\vec{x})$ of $f$ at $\vec{x}\in M$ is defined to be the unique vector in $T_{\vec{x}}M$ satisfying $Df(\vec{x})[\vec{s}]=\langle\grad f(\vec{x}),\vec{s}\rangle_{\vec{x}}$, where $Df(\vec{x})[\vec{s}]$ is the directional derivative of $f$ at $\vec{x}$ along $\vec{s}$.

\paragraph{Retraction} A \emph{retraction} on a manifold $M$ is a smooth mapping $\Retr_{\vec{x}}$ satisfying following properties:
\begin{enumerate}
\item $\Retr_{\vec{x}}(\vec{0})=\vec{x}$, where $\vec{0}$ is the zero vector in $T_{\vec{x}}M$.
\item The differential of $\Retr_{\vec{x}}$ at $\vec{0}$ is the identity map.
\end{enumerate}
Then the Riemannian gradient descent with stepsize $\eta$ is defined to be
\[
\vec{x}_{k+1}=\Retr_{\vec{x}_k}(-\eta\grad f(\vec{x}_k)).
\]
\begin{assumption}
There exist $b>0$ and $L>0$ such that for all $\vec{x}\in M$ and $\vec{s}\in T_{\vec{x}}M$ with $\norm{\vec{s}}<b$, 
$
\norm{\nabla \hat{f}_{\vec{x}}(\vec{s})-\nabla\hat{f}_{\vec{x}}(\vec{0})}\le L\norm{\vec{s}},
$
where $\hat{f}_{\vec{x}}=f\circ\Retr_{\vec{x}}$.
\end{assumption}

\begin{remark}
This Lipschitz type assumption (from \cite{CB}) gives criterion of the step-size used in the following theorem.
\end{remark}

The following theorem is about the existence of invariant functions for Riemannian gradient descent.

\begin{theorem}\label{IF:manifold}
Let $f:M\rightarrow\mathbb{R}$ be a $C^2$ function. The Riemannian gradient descent is 
\[
T(\vec{x}_k)=\Retr_{\vec{x}_k}(-\eta\grad f(\vec{x}_k)). 
\]
Then for step-size $\eta<1/L$ and any continuous function $p(\vec{x})$ on $M$,
\[
\Phi(\vec{x})=\sum_{-\infty}^{\infty}p(T^n(\vec{x}))(f(T^{n-1}\vec{x})-f(T^n(\vec{x})))
\]
is $T$-invariant on an open dense subset of $M-F$, where $F$ is the set of fixed points of $T$.
\end{theorem}
\subsection{Classification for Orbits}
 We now give a further explanation on the potential applications of the abundant invariant functions implied by Theorem \ref{invariant functions:GD} and \ref{IF:compact}. As mentioned before, if $T$ is an update rule on a $d$-dimensional space, then the set of orbits can be equivalently described by values of the $d$ invariant functions. Throughout this section, we denote $M$ the Euclidean space or the product of simplices. With the aid of these invariant functions, we can define an oracle classifier to tell if two points $\vec{x}$ and $\vec{y}$ are on the same orbit. We formulate this in the following proposition.

\begin{proposition}\label{applicationxxx}
Suppose function $f:M\rightarrow\mathbb{R}$ be a differentiable function defined on a $d$-dimensional space $M$. Let $T$ be the optimization algorithm that is a global homeomorphism on $M$. Then there is a mapping $\varphi(\vec{x}):M-F\rightarrow M$ such that $\vec{x},\vec{y}\in M$ are on the same orbit if and only if $\varphi(\vec{x})=\varphi(\vec{y})$, where $F$ is the set of fixed points of $T$.
\end{proposition}

\begin{proof}
Define $\varphi(\vec{x})=(\varphi_1(\vec{x}),...,\varphi_d(\vec{x}))$ where $\varphi_i(\vec{x})$ for $i\in[d]$ are the functions proven to exist in Theorem \ref{invariant functions:GD}. 
 There exist invariant functions $\varphi_1(\vec{x}),...,\varphi_d(\vec{x})$. Denote $\varphi(\vec{x})=(\varphi_1(\vec{x}),...,\varphi_d(\vec{x}))$, from the construction of $\varphi(\vec{x})$, we know that $\varphi(\vec{x})$ is the unique intersection of $\{T^n(\vec{x})\}_{n\in\mathbb{Z}}$. So $\varphi(\vec{x})=\varphi(\vec{y})$ if and only if $\vec{x}$ and $\vec{y}$ belong to the same orbit, furthermore, this means that $\varphi_i(\vec{x})=\varphi_i(\vec{y})$ for all $i\in[d]$ if and only if $\vec{x}$ and $\vec{y}$ belong to the same orbit.
\end{proof}

\section{Conclusion}
In this paper, we study the generic existence of invariant functions for first-order methods in non-convex optimization settings and games. For the easier, (optimization setting), we prove that the existence of a maximal number of invariant functions (as many as the dimension of the space). 
 For the case games, e.g. two player/network  network coordination games, we compute  these invariant functions in closed form.  
 These invariant functions effectively constrain the system away from uncontrollable and chaotic behavior.
  Thus, our results present a middle, and largely unexplored, ground between totally convergent dynamics (e.g. via the use of a Lyapunov/potential function) and total unpredictability of optimization driven dynamics. 

\bibliography{ms,bibli,refer,refer2}
\bibliographystyle{plain}
\newpage
\appendix
\section{Missing proofs of Section 4}

\noindent
\textbf{Proof of Theorem \ref{bipartitegame}}
\begin{proof}
Inspired by the proof of Lemma 1 and Lemma 7 in \cite{BGP}, we complete the proof in the following way. 
Denote $\vec{A}=(A^{ij})$ the block matrix with blocked entries the matrices $A^{ij}$. Notice that 
\[
\norm{X}^2=\sum_{i=1}^n\norm{\vec{x}_i}^2, \ \ \text{and}\ \ \norm{Y}^2=\sum_{j=1}^m\norm{\vec{y}_j}^2.
\]
It suffices to show that
\begin{align}
\frac{\norm{X^{t+1}}^2}{\eta_1}-\frac{\norm{Y^{t+1}}^2}{\eta_2}+\langle X^{t+1},\vec{A}Y^{t+1}\rangle
=\frac{\norm{X^t}^2}{\eta_1}-\frac{\norm{Y^t}^2}{\eta_2}+\langle X^t,\vec{A}Y^t\rangle.
\end{align}

Subtracting $\norm{X^t}^2$ from $\norm{X^{t+1}}^2$ and $\norm{Y^t}^2$ from $\norm{Y^{t+1}}^2$, we have
\begin{align}
\norm{X^{t+1}}^2-\norm{X^t}^2
&=\norm{X^t+\eta_1\vec{A}Y^t}^2-\norm{X^t}^2
\\
&=\langle X^t+\eta_1\vec{A}Y^t,X^t+\eta_1\vec{A}Y^t\rangle-\norm{X^t}^2
\\
&=\langle\eta_1\vec{A}Y^t,\eta_1\vec{A}Y^t\rangle+2\eta_1\langle X^t,\vec{A}Y^t\rangle
\\
&=\eta_1\langle X^{t+1}-X^t,\vec{A}Y^t\rangle+2\eta_1\langle X^t,\vec{A}Y^t\rangle
\\
&=\eta_1\langle X^{t+1}+X^t,\vec{A}Y^t\rangle.
\end{align}

and
\begin{align}
\norm{Y^{t+1}}^2-\norm{Y^t}^2
&=\norm{Y^t+\eta_2\vec{A}^{\top}X^{t+1}}^2-\norm{Y^t}^2
\\
&=\norm{Y^t}^2+2\langle Y^t,\eta_2\vec{A}^{\top}X^{t+1}\rangle
+\norm{\eta_2\vec{A}^{\top}X^{t+1}}^2-\norm{Y^t}^2
\\
&=\eta_2\langle \vec{A}^{\top}X^{t+1},\eta_2\vec{A}^{\top}X^{t+1}\rangle
+2\eta_2\langle Y^t,\vec{A}^{\top}X^{t+1}\rangle
\\
&=\eta_2\langle \vec{A}^{\top}X^{t+1},Y^{t+1}-Y^t\rangle
+2\eta_2\langle X^{t+1},\vec{A}Y^t\rangle
\\
&=\eta_2\langle\vec{A}^{\top}X^{t+1},Y^{t+1}\rangle
-\eta_2\langle \vec{A}^{\top}X^{t+1},Y^{t}\rangle
+2\eta_2\langle X^{t+1},\vec{A}Y^{t}\rangle
\\
&=\eta_2\langle\vec{A}^{\top}X^{t+1},Y^{t+1}\rangle
+\eta_2\langle X^{t+1},\vec{A}Y^t\rangle
\\
&=\eta_2\langle X^{t+1},\vec{A}(Y^{t+1}+Y^t)\rangle.
\end{align}
The difference
\begin{align}
\frac{\norm{X^{t+1}}^2-\norm{X^t}^2}{\eta_1}-\frac{\norm{Y^{t+1}}^2-\norm{Y^t}^2}{\eta_2}
=\langle X^t,\vec{A}Y^t\rangle-\langle X^{t+1},\vec{A}Y^{t+1}\rangle
\end{align}
implies
\begin{align}
\frac{\norm{X^{t+1}}^2}{\eta_1}-\frac{\norm{Y^{t+1}}^2}{\eta_2}+\langle X^{t+1},\vec{A}Y^{t+1}\rangle
=\frac{\norm{X^t}^2}{\eta_1}-\frac{\norm{Y^t}^2}{\eta_2}+\langle X^t,\vec{A}Y^t\rangle.
\end{align}
The proof completes.
\end{proof}

\section{Missing proofs of Section 5}
We next finish the proof of Theorem \ref{invariant functions:GD}. Before proving the technical lemmas, we give a formal definition of "Fundamental Set" which plays an essential role in the proof.

\begin{definition}[Fundamental set, \cite{Stebe}]\label{def:fundset}
A \emph{fundamental set} $S$ for $T$ on $M$ is a subset of $M$ satisfying the following: $S$ contains no fixed point of $T$ but if $\vec{x}$ is not a fixed point of $T$, $T^n(\vec{x})\in S$ for a single integer $n$ depending on $S$ and $\vec{x}$.
\end{definition}
Roughly speaking, a fundametal set is a single piece in $M$ such that each orbit $\{T^n(\vec{x})\}$ meets $S$ at most once. 

\begin{lemma}\label{lemma:IF of GD}
Let the stepsize $\eta<\frac{2}{dL}$. Then $f(\vec{x}-\eta\nabla f(\vec{x}))<f(\vec{x})$ for all $\vec{x}\in \mathbb{R}^d$ and $\vec{x}$ is not a fixed point.
\end{lemma}

\begin{proof}
According to Taylor's theorem \ref{Taylor}, letting
$
\vec{a}=\vec{x} \ \ \ \text{and}\ \ \ \vec{h}=-\eta\nabla f(\vec{x}),
$
 we have 
 $
 \sum_{|\alpha|\le 1}\frac{\partial^{\alpha}f(\vec{a})}{\alpha!}\vec{h}^{\alpha}=f(\vec{x})-\eta\sum_{i=1}^d\left(\frac{\partial f}{\partial x_i}(\vec{x})\right)^2
 $
 and 
 $
 R_{\vec{a},1}(\vec{h})=R_{\vec{x},1}(-\eta\nabla f(\vec{x})).
 $
 Notice that
 \begin{align}
 R_{\vec{x},1}(-\eta\nabla f(\vec{x}))&\le \abs{R_{\vec{x},1}(-\eta\nabla f(\vec{x}))}
 \\
 &=\abs{\sum_{|\alpha|=2}\partial^{\alpha}f(\vec{x}-c\eta\nabla f(\vec{x}))\frac{(-\eta\nabla f(\vec{x}))^{\alpha}}{\alpha!}}
 \\
 &\le \sum_{|\alpha|=2}\abs{\partial^{\alpha}f(\vec{x}-c\eta\nabla f(\vec{x}))\frac{(-\eta\nabla f(\vec{x}))^{\alpha}}{\alpha!}}
 \\
 &= \sum_{|\alpha|=2}\abs{\partial^{\alpha}f(\vec{x}-c\eta\nabla f(\vec{x}))}\cdot\abs{\frac{(-\eta\nabla f(\vec{x}))^{\alpha}}{\alpha!}}
 \\
 &\le L\sum_{|\alpha|=2}\abs{\frac{(-\eta\nabla f(\vec{x}))^{\alpha}}{\alpha!}} 
 \\
 &=L\eta^2\sum_{|\alpha|=2}\abs{\frac{(\nabla f(\vec{x}))^{\alpha}}{\alpha!}}
 \end{align}
 \begin{align}
 &=\frac{L\eta^2}{2}\left(\sum_{i=1}^d\abs{\frac{\partial f}{\partial x_i}(\vec{x})}\right)^2,
 \end{align}
 where the last equality holds since 
 \begin{align}
 \sum_{|\alpha|=2}\abs{\frac{(\nabla f(\vec{x}))^{\alpha}}{\alpha!}}&=\frac{1}{2}\sum_{i=1}^d\left(\frac{\partial f}{\partial x_i}(\vec{x})\right)^2+\sum_{i<j}\abs{\frac{\partial f}{\partial x_i}(\vec{x})}\cdot\abs{\frac{\partial f}{\partial x_j}(\vec{x})}
 \\
 &=\frac{1}{2}\left(\sum_{i=1}^d\abs{\frac{\partial f}{\partial x_i}(\vec{x})}\right)^2.
 \end{align}

 To prove $f(\vec{x}-\eta\nabla f(\vec{x}))<f(\vec{x})$, by Taylor's theorem, it is equivalent to prove that
\begin{align}
f(\vec{x}-\eta\nabla f(\vec{x}))&=f(\vec{x})+\nabla f(\vec{x})\cdot(-\eta\nabla f(\vec{x}))+R_{\vec{x},1}(-\eta\nabla f(\vec{x}))
\\
&=f(\vec{x})-\eta\norm{\nabla f(\vec{x})}^2+R_{\vec{x},1}(-\eta\nabla f(\vec{x}))
\\
&<f(\vec{x}).
\end{align}
 From the above arguments, it suffices to show that 
 \begin{equation}\label{ineq}
 -\eta\norm{\nabla f(\vec{x})}^2+\frac{L\eta^2}{2}\left(\sum_{i=1}^d\abs{\frac{\partial f}{\partial x_i}(\vec{x})}\right)^2<0.
 \end{equation}
 Since $l^1$ norm is equivalent to $l^2$ norm on $\mathbb{R}^d$, i.e. $\sum v_i^2\le(\sum \abs{v_i})^2\le d\sum v_i^2$, we have that
 \[
 \frac{1}{d}\le\frac{\sum_{i=1}^d\left(\frac{\partial f}{\partial x_i}(\vec{x})\right)^2}{\left(\sum_{i=1}^d\abs{\frac{\partial f}{\partial x_i}(\vec{x})}\right)^2}.
 \]
 Then the inequality \ref{ineq} follows provided $\eta<\frac{2}{dL}$, and thus the proof completes.
\end{proof}

\begin{lemma}\label{fundlemma1}
 For each $\vec{x}_0\in\mathbb{R}^d-F$, there exists a neighborhood $N$ of $\vec{x}_0$, such that $T^r(N)\cap T^s(N)=\emptyset$ if $r\ne s$.
\end{lemma}
\begin{proof}
The proof is completed in three steps.
\\
1. For each $x_0\in \mathbb{R}^d-F$, there exists $N$ such that $T^{-1}(N)\cap N=\emptyset$.
\\
Since $f(T^{-1}(\vec{x}_0))-f(\vec{x}_0)=\Delta>0$ (clearly it is from that $f(T^{-1}(\vec{x}))>f(TT^{-1}(\vec{x}))$), $f$ is continuous, so there exists neighborhood $U$ of $\vec{x}_0$ such that $f(\vec{x})<f(\vec{x}_0)+\frac{\Delta}{3}$ for all $\vec{x}\in U$, and neighborhood $V$ of $T^{-1}(\vec{x}_0)$ such that $f(\vec{y})>f(T^{-1}(\vec{x}_0))-\frac{\Delta}{3}$ for all $\vec{y}\in V$. We have that $T(V)\cap U$ is open since $T^{-1}$ is homeomorphism. Let $N$ be a neighborhood of $\vec{x}$ in $T(V)\cap U$, then $N\subset U$ and $T^{-1}(N)\subset V$. We have 
\[
f(\vec{x})<f(\vec{x}_0)+\frac{\Delta}{3}<f(T^{-1}(\vec{x}_0))-\frac{\Delta}{3}<f(\vec{y}) \ \ \ \text{for all}\ \ \ \vec{x}\in N, \ \ \vec{y}\in V.
\]
So $T^{-1}(N)\cap N=\emptyset$.
\\
\\
2. $T^{-m}(N)\cap N=\emptyset$ for all $m\ge 1$.
\\
Suppose $\vec{x}\in N$ and $\vec{z}\in T^{-m}(N)$, then $\vec{z}=T^{-m}(\vec{u})$ for some $\vec{u}\in N$, and then 
\[
f(\vec{z})=f(T^{-m}(\vec{u}))>f(T^{-1}(\vec{u}))>f(\vec{x}),
\]
since $T^{-1}(\vec{u})\in T^{-1}(N)$. This shows that $\vec{z}\ne \vec{x}$.
\\
\\
3. $T^{-r}(N)\cap T^{-s}(N)=\emptyset$ if $r\ne s$.
\\
Assume $r>s$, let $\vec{y}\in T^{-r}(N)\cap T^{-s}(N)$. Then $T^{r}(\vec{y})\in N$ and $T^{-r+s}(T^{r}(\vec{y}))=T^{s}(\vec{y})\in N$. So $T^{r}(\vec{y})\in N\cap T^{-r+s}(N)$ which is impossible from 2.
\end{proof}

\begin{lemma}\label{fundlemma2}
$T:\mathbb{R}^d\rightarrow \mathbb{R}^d$. There is an $\epsilon>0$, such that for all $\vec{x}\in \mathbb{R}^d-F$, there exists $n$ such that $dist(T^n(\vec{x}),F)\ge \epsilon$.
\end{lemma}
\begin{proof}
According to Lemma 2.3 \cite{Stebe}, the set of fixed points $F$ is asymptotically stable, and this lemma follows from the combined arguments with Theorem 4.19 of \cite{BhS}.
\end{proof}

\begin{lemma}\label{fundlemma3}
Let $T:\mathbb{R}^d\rightarrow \mathbb{R}^d$ be a homeomorphism, then $T$ has a fundamental set.
\end{lemma}
\begin{proof}
Let $D_{\epsilon}=\{\vec{x}\in \mathbb{R}^d:dist(x,F)\ge\epsilon\}$, where $\epsilon\ge 0$ may be chosen so that $D_{\epsilon}\ne\emptyset$. Since $D_{\epsilon}\cap F=\emptyset$, from Lemma \ref{fundlemma1}, for each $\vec{x}\in D_{\epsilon}$, there exists a neighborhood $N_{\vec{x}}$ of $\vec{x}$ such that $T^n(N_{\vec{x}})$ are disjoint. Since $\mathbb{R}^d$ is a second countable space, $D_{\epsilon}$ has countable basis, i.e. it has countable open cover $N_1,...,N_r,...$. Then each sequence $T^n(\vec{x})$ for $\vec{x}\in \mathbb{R}^d-F$, $T^n(\vec{x})$ only meets $N_i$ once.
\\
Let 
\begin{align}
L_1&=N_1
\\
L_2&=N_2-\cup_{-\infty}^{\infty}T^n(N_1)
\\
L_3&=N_3-\cup_{-\infty}^{\infty}T^n(N_1)-\cup_{-\infty}^{\infty}T^n(N_2)
\\
&\vdots
\\
L_r&=N_r-\cup_{-\infty}^{\infty}T^n(N_1)-...-\cup_{-\infty}^{\infty}T^n(N_{r-1}),
\\
&\vdots
\end{align}
then $S=\cup_{i=1}^{\infty} L_i$ is a fundamental set.
\end{proof}

\begin{lemma}\label{fundlemma4}
Let $T$ be the gradient descent for $f:\mathbb{R}^d\rightarrow\mathbb{R}$, and $F$ be the set of fixed points of $T$, there exist d $T$-invariant functions of $T$ which are continuous and independent on an open dense subset of $\mathbb{R}^d-F$. 
\end{lemma}
\begin{proof}
Let $S\subset\mathbb{R}^d$ be a fundamental set for $T$. Let $\partial S$ be the boundary of $S$ and let $B=\cup_{-\infty}^{\infty}T^n(\partial S)$. Then $\mathbb{R}^d-F-B$ is dense in $\mathbb{R}^d-F$.
Denote $M'=\mathbb{R}^d-F-B$ and recall that
\[
[\vec{x}]=\{T^n(\vec{x}):n\in\mathbb{Z}\}
\]
and 
\[
M'/T:=\{[\vec{x}]:\vec{x}\in M'\}.
\]
Let $\varphi(\vec{x})$ be the element of $\{T^n(\vec{x})\}$ in $S$. We next show that it is continuous. 
\\
If $\vec{x}\in M'$, $\varphi(\vec{x})$ is the unique intersection of $\{T^n(\vec{x})\}$ with $S$. Hence there is an integer $m$ such that $T^m(\vec{x})\in S$. 
Let $U$ be a neighborhood of $T^m(\vec{x})$ in $S$. Since $T^m$ is continuous, $V=(T^m)^{-1}(U)=T^{-m}(U)$ is a neighborhood of $\vec{x}$. If $\vec{y}\in V$, $T^m(\vec{y})\in S$ so that $\varphi(\vec{y})=T^m(\vec{y})$ for all $\vec{y}\in V$. Hence $\varphi$ is continuous in a neighborhood of $\vec{x}\in M'$, $M'$ is open and $\varphi=T^m$ for some $m$ in a neighborhood of $\vec{x}\in \mathbb{R}^d$.
\\
We set 
\[
\varphi(\vec{x})=(\varphi_1(\vec{x}),...,\varphi_d(\vec{x}))
\] 
so that the $\varphi_i(\vec{x})$ are the components of $\varphi(\vec{x})$, it follows that the $\varphi_i(\vec{x})$ are continuous and independent on $M'$, since $\varphi(\vec{x})$ is a local homeomorphism on $M'$. Since $\varphi(T(\vec{x}))=\varphi(\vec{x})$, $\varphi_i(T(\vec{x}))=\varphi_i(\vec{x})$, meaning that $\varphi_i(\vec{x})$ are $T$-invariant, and then the proof is complete.
\end{proof}

\noindent
\textbf{Proof of Theorem \ref{invariant functions:GD}}
\begin{proof}
With the above lemmas, we can finish the proof by showing that $T$ is a local diffeomorphism when the stepsize $\eta$ is chosen small enough, which boils down to showing that the Jacobian matrix of $T(\vec{x})=\vec{x}-\eta\nabla f(\vec{x})$ is non-singular when $\eta$ approaching $0$, i.e. $\det DT(\vec{x})\ne 0$ as $\eta\rightarrow 0$. Specifically, 
\[
DT(\vec{x})=I-\eta\nabla^2f(\vec{x}).
\]
Since by assumptions, all the entries of $\nabla^2f(\vec{x})$ are uniformly bounded, the determinant $\det DT(\vec{x})$ is continuous with respect to its coefficients, and $\lim_{\eta\rightarrow 0} DT(\vec{x})=I$, so $\det DT(\vec{x})\rightarrow 1$ as $\eta\rightarrow 0$, which means $DT(\vec{x})$ is invertible. 
By the assumptions, $T$ is a proper map. Since the Euclidean space $\mathbb{R}^d$ is simply-connected (fundamental group is trivial since $\mathbb{R}^d$ is homotopic to a point), by Theorem \ref{proper map} (Theorem 2 of \cite{Ho1975}), the fact that $T$ is a local diffeomorphism and proper map on simply-connected space $\mathbb{R}^d$ implies that $T$ is a global homeomorphism, i.e. $T$ is invertible on $\mathbb{R}^d$ so that $T^{-1}$ is well defined. Then from Lemma \ref{fundlemma4}, there exists $d$ continuous and independent $T$-invariant functions on an open dense subset of $\mathbb{R}^d-F$ where $F$ is the set of fixed points of $T$.
\end{proof}

\begin{lemma}\label{IF:sum1}
The set of cluster points of $\{T^n(\vec{x})\}_{n\in\mathbb{Z}}$ is the union of $L_{\vec{x}}$ and $l_{\vec{x}}$. The value of $f$ is constant on each of $L_{\vec{x}}$ and $l_{\vec{x}}$. If $f(L_{\vec{x}})$ denotes the value of $f$ on $l_{\vec{x}}$ and $f(L_{\vec{x}})$ denotes the value of $f$ on $L_{\vec{x}}$ we have $f(L_{\vec{x}})>f(l_{\vec{x}})$ whenever $\vec{x}$ is not a fixed point of $T$ in $\mathbb{R}^d$.
\end{lemma}
\begin{proof}
Let $\vec{a},\vec{b}\in l_{\vec{x}}$, then by definition of cluster point, we have two subsequences $\{n_i\}_{\mathbb{Z}_+}$ and $\{n_j\}_{\mathbb{Z}_+}$ of $n\ge 0$, such that
\[
\lim_{i\rightarrow\infty}\norm{T^{n_i}(\vec{x})-\vec{a}}=0
\]
and
\[
\lim_{j\rightarrow\infty}\norm{T^{n_j}(\vec{x})-\vec{b}}=0.
\]
Since $f$ is continuous, we have that 
\[
\lim_{i\rightarrow\infty}f(T^{n_i}(\vec{x}))=f(\vec{a})
\]
and 
\[
\lim_{j\rightarrow\infty}f(T^{n_j}(\vec{x}))=f(\vec{b}).
\]
By the fact that 
\begin{align}
\lim_{i\rightarrow\infty}f(T^{n_i}(\vec{x}))&=\lim_{i\rightarrow\infty}f(T^{n_i}(\vec{x}))
\\
&=\text{local minimum with initial condition}\ \ \vec{x},
\end{align} 
we conclude $f(\vec{a})=f(\vec{b})$.
\\
The other case, let $\{n\}=\{n=-1,-2...\}$ be the sequence of negative integers and $\vec{c},\vec{d}\in L_{\vec{x}}$, there exist two subsequences negative integers $\{n_i\}_{i\in\mathbb{Z}_+}$ and $\{n_j\}_{j\in\mathbb{Z}_+}$ of $\{n\}$, such that
\[
\lim_{i\rightarrow\infty}\norm{T^{n_i}(\vec{x})-\vec{c}}=0
\]
and
\[
\lim_{j\rightarrow\infty}\norm{T^{n_j}(\vec{x})-\vec{d}}=0,
\]
and by the continuity of $f$, we have
\[
\lim_{i\rightarrow\infty}f(T^{n_i}(\vec{x}))=f(\vec{c})
\]
and 
\[
\lim_{j\rightarrow\infty}f(T^{n_j}(\vec{x}))=f(\vec{d}).
\]
Since 
\begin{align}
\lim_{i\rightarrow\infty}f(T^{n_i}(\vec{x}))&=\lim_{j\rightarrow\infty}f(T^{n_j}(\vec{x}))
\\
&=\text{local maximum with initial condition}\ \ \vec{x},
\end{align} 
we have $f(\vec{c})=f(\vec{d})$.
\end{proof}

\begin{lemma}\label{B2x}
Let $\vec{x}_0$ be an element of $\mathbb{R}^d$. Either there is a neighborhood $N$ of $\vec{x}_0$ such that $f(L_{\vec{x}})=f(L_{\vec{x}_0})$ for all $\vec{x}\in N$ or in every neighborhood of $\vec{x}_0$ there is an $\vec{x}$ such that $f(L_{\vec{x}})>f(L_{\vec{x}_0})$. 
\end{lemma}
\begin{proof}
Suppose there is a neighborhood $N_1$ of $\vec{x}_0$ in $\mathbb{R}^d$ such that $f(L_{\vec{x}_0})\ge f(L_{\vec{x}})$ for all $x\in N_1$. Let $\xi$ be a positive number. Let $S_{\xi}=\{\vec{x}:f(L_{\vec{x}})\ge f(L_{\vec{x}_0})-\xi\}$. We show that $S_{\xi}$ is open. If $\vec{x}$ is an element of $S_{\xi}$, there is an $m$ such that $f(T^m(\vec{x}))>f(L_{\vec{x}_0})-\xi$. Since $T^m$ is continuous, there is a neighborhood $N_{\vec{x}}$ of $\vec{x}$ such that $f(T^m(\vec{y}))>f(L_{\vec{x}_0})-\xi$ for all $\vec{y}$ in $N_{\vec{x}}$. But $f(L_{\vec{y}})\ge f(T^m(\vec{y}))$ for all $\vec{y}\in\mathbb{R}^d$ so that $f(L_{\vec{y}})\in S_{\xi}$ for all $y\in N_{\vec{x}}$. Hence $S_{\xi}$ is open. Let $N(\xi)=S_{\xi}\cap N_{\vec{x}_0}$. Since $\vec{x}_0$ is an element of $S_{\xi}$ for all positive $\xi$, $N(\xi)$ is not empty for $\xi>0$. Since $N(\xi)$ is contained in $N_{\vec{x}_0}$ and $S_{\xi}$, $f(L_{\vec{x}_0})\ge f(L_{\vec{x}})\ge f(L_{\vec{x}_0})-\xi$ for all $\vec{x}$ in $N(\xi)$. Since the points of $L_{\vec{x}}$ are in $F$, the set of fixed points of $T$, $f(L_{\vec{x}})$ can assume only finitely many values. Hence for $\xi$ sufficiently small
\[
f(L_{\vec{x}_0})\ge f(L_{\vec{x}})\ge f(L_{\vec{x}_0})-\xi
\]
implies that $f(L_{\vec{x}})=f(L_{\vec{x}_0})$, and so for some $\xi$, $\vec{x}\in N(\xi)$ implies that $f(L_{\vec{x}})=f(L_{\vec{x}_0})$.
\end{proof}

\begin{lemma}\label{IF:sum3}
Let $\vec{x}_0$ be an element of $\mathbb{R}^d$. Either there is a neighborhood $N_{\vec{x}_0}$ of $\vec{x}_0$ in $\mathbb{R}^d$ such that $f(L_{\vec{x}_0})=f(L_{\vec{x}})$ for all $\vec{x}$ in $N_{\vec{x}_0}$ or every neighborhood $N$ of $\vec{x}_0$ contains an open subset $V_{N}$ such that $f(L_{\vec{y}})=f(L_{\vec{z}})$ for all $\vec{y}$ and $\vec{z}$ in $V_{N}$.
\end{lemma}
\begin{proof}
Suppose $\vec{x}_0$ is an element of $\mathbb{R}^d$ and there is no neighborhood $U$ of $\vec{x}_0$ in $\mathbb{R}^d$ such that $f(L_{\vec{x}})=f(L_{\vec{x}_0})$ for all $\vec{x}$ in $U$. Let $N$ be a neighborhood of $\vec{x}_0$. According to the lemma \ref{B2x}, there is an element $\vec{x}$ of $N$ such that $f(L_{\vec{x}})>f(L_{\vec{x}_0})$. Let $K$ be the least upper bound of $f(L_{\vec{x}})$ for $\vec{x}$ in $N$. Since the range of $f(L_{\vec{x}})$ is finite, there is a point $\vec{y}$ of $N$ such that $f(L_{\vec{y}})=K$. Thus $f(L_{\vec{y}})\ge f(L_{\vec{x}})$ for all $\vec{x}$ in $N$, and $N$ is a neighborhood of $\vec{y}$. By lemma \ref{B2x}, there is a neighborhood $U$ of $\vec{y}$ such that $f(L_{\vec{y}})=f(L_{\vec{x}})$ for all $\vec{x}\in U$. Let $V_{N}=N\cap U$.
\end{proof}

\begin{lemma}\label{IF:sum4}
Let $\vec{x}_0$ be an element of $\mathbb{R}^d$. Either there is a neighborhood $N_{\vec{x}_0}$ of $\vec{x}_0$ in $\mathbb{R}^d$ such that $f(l_{\vec{x}_0})=f(l_{\vec{x}})$ for all $\vec{x}$ in $N_{\vec{x}_0}$, or every neighborhood $N$ of $\vec{x}_0$ contains an open subset $U_{N}$ such that $f(l_{\vec{y}})=f(l_{\vec{z}})$ for all $\vec{y}$ and $\vec{z}$ in $U_N$.
\end{lemma}
\begin{proof}
Using the fact that if $T$ is a homeomorphism of $\mathbb{R}^d$ onto itself, $T^{-1}$ is defined and either $x=T^{-1}(x)$ or $f(T^{-1}(\vec{x}))<f(\vec{x})$, we can modify the above arguments by replacing $T$ with $T^{-1}$ and reversing the inequalities to have the results about the function $f(l_{\vec{x}})$. Suppose there is a neighborhood $N_1$ of $\vec{x}_0$ in $\mathbb{R}^d$ such that $f(l_{\vec{x}_0})\le f(l_{\vec{x}})$ for all $\vec{x}\in N_1$. Let $\xi$ be a positive number. Let $S_{\xi}=\{\vec{x}:f(l_{\vec{x}})\le f(l_{\vec{x}_0})+\xi\}$. We show that $S_{\xi}$ is open. If $\vec{x}$ is an element of $S_{\xi}$, there is an $m$ such that $f(T^m(\vec{x}))<f(l_{\vec{x}_0})+\xi$. Since $T^m$ is continuous, there is a neighborhood $N_{\vec{x}}$ of $\vec{x}$ such that $f(T^m(\vec{y}))<f(l_{\vec{x}_0})+\xi$ for all $\vec{y}$ in $N_{\vec{x}}$. But $f(l_{\vec{y}})\le f(T^m(\vec{y}))$ for all $\vec{y}\in\mathbb{R}^d$ so that $f(l_{\vec{y}})\in S_{\xi}$ for all $\vec{y}\in N_{\vec{x}}$. Hence $S_{\xi}$ is open. Let $N(\xi)=S_{\xi}\cap N_{\vec{x}_0}$. Since $\vec{x}_0$ is an element of $S_{\xi}$ for all positive $\xi$, $N(\xi)$ is not empty for $\xi>0$. Since $N(\xi)$ is contained in $N_{\vec{x}_0}$ and $S_{\xi}$, $f(l_{\vec{x}_0})\le f(l_{\vec{x}})\le f(l_{\vec{x}_0})+\xi$ for all $\vec{x}$ in $N(\xi)$. Since the points of $l_{\vec{x}}$ are in the fixed point set $F$, $f(l_{\vec{x}})$ can assume only finitely many values. Hence for $\xi$ sufficiently small 
\[
f(l_{\vec{x}_0})\le f(l_{\vec{x}})\le f(l_{\vec{x}_0})+\xi
\]
implies that $f(l_{\vec{x}})=f(l_{\vec{x}_0})$, and so for some $\xi$, $\vec{x}\in N(\xi)$ implies that $f(l_{\vec{x}})=f(l_{\vec{x}_0})$. Next, suppose $\vec{x}_0$ is an element of $\mathbb{R}^d$ and there is no neighborhood $U$ of $\vec{x}_0$ in $\mathbb{R}^d$ such that $f(l_{\vec{x}})=f(l_{\vec{x}_0})$ for all $\vec{x}$ in $U$. Let $N$ be a neighborhood of $\vec{x}_0$. According to the above arguments, there is an element $\vec{x}$ of $N$ such that $f(l_{\vec{x}})<f(l_{\vec{x}_0})$. Let $K$ be the least upper bound of $f(L_{\vec{x}})$ for $\vec{x}$ in $N$. Since the range of $f(l_{\vec{x}})$ is finite, there is a point $\vec{y}$ of $N$ such that $f(l_{\vec{y}})=K$. Thus $f(l_{\vec{y}})\le f(l_{\vec{x}})$ for all $\vec{x}$ in $N$, and $N$ is a neighborhood of $\vec{y}$. Thus there is a neighborhood $U$ of $\vec{y}$ such that $f(l_{\vec{y}})=f(l_{\vec{x}})$ for all $\vec{x}\in U$. Let $U_N=N\cap U$, the proof completes.
\end{proof}

\noindent
\textbf{Proof of Theorem \ref{IF:construction}}
\begin{proof}
Next we complete the proof of theorem. For each $\vec{x}\in \mathbb{R}^d-F$, $\Phi(\vec{x})$ is convergent. Let $G_1$ be the set of all elements $\vec{x}$ of $\mathbb{R}^d$ such that $f(L_{\vec{x}})$ is constant in a neighborhood of $\vec{x}$. Let $G_2$ be the set of all elements $\vec{x}$ of $\mathbb{R}^d$ such that $f(l_{\vec{x}})$ is a constant in a neighborhood of $\vec{x}$. Notice that $G=(\mathbb{R}^d-F)\cap G_1\cap G_2$ is an open dense subset of $\mathbb{R}^d-F$. For each $\vec{x}\in \mathbb{R}^d$, let 
\[
S(\vec{x})=\sum_{n=-\infty}^{\infty}f(T^{n-1}(\vec{x}))-f(T^n(\vec{x})).
\]
Clearly $S(\vec{x})$ converges at each $\vec{x}$ to $f(L_{\vec{x}})-f(l_{\vec{x}})$. Let $\vec{y}$ be an element of $G$. There is a neighborhood $U$ of $\vec{y}$ such that $S(\vec{x})$ represents the constant function in $U$. Since $\vec{y}\notin F$ and $F$ is compact, there is a neighborhood $W$ of $\vec{y}$ such that $\bar{W}\subset U\cap V$. Then $S(\vec{x})$ is a series of positive terms converging to a continuous function on $\bar{W}$ and $S(\vec{x})$ converges uniformly on $\bar{W}$. Let $p(\vec{x})$ be any bounded function continuous on $\mathbb{R}^d$. The series
\[
\Phi(\vec{x})=\sum_{n=-\infty}^{\infty}p(T^n(\vec{x}))\left(f(T^{n-1}(\vec{x}))-f(T^n(\vec{x}))\right)
\]
converges uniformly on $\bar{W}$ since $f$ is taken to be bounded on $\mathbb{R}^d$. Since $p$, $f$ and $T$ are continuous, $\Phi(\vec{x})$ is continuous on $\bar{W}$ and hence at $\vec{y}$. The invariance is obvious, so the proof completes.
\end{proof}

\noindent
\textbf{Proof of Corollary \ref{IF:construction discrete crit}}
We first recall the Center-Stable Manifold Theorem.
\begin{theorem}[Center-Stable Manifold Theorem, \cite{shub}]\label{Center-Stable Manifold Theorem}
Let $0$ be a fixed point for the $C^r$ local diffeomorphism $\phi : U \to E$, where $U$ is a neighborhood of $0$ in the Banach space $E$. Suppose that $E = E_s \oplus E_u$, where $E_u$ is the span of the eigenvectors corresponding to eigenvalues less than or equal to $1$ of $D\phi(0)$, and $E_u$ is the span of eigenvalues greater than $1$ of $D\phi(0)$. Then there exist a $C^r$ embedded disk $W^{cs}_{loc}$ that is tangent to $E_s$ at $0$ called the local stable center manifold. Moreover, there exist a neighborhood $B$ of $0$ such that $\phi(W^{cs}_{loc}) \cap B \subset W^{cs}_{loc}$, and $\cap^{\infty}_{k=0} \phi^{-k}(B) \subset W^{cs}_{loc}$.
\end{theorem}
The theorem applies to a very general framework where the underlying space is Banach manifold. In our proof, we focus on the Euclidean space and this theorem implies a simple fact for the dynamical system defined by the gradient descent: at each strict saddle point of $f$, there exist stable and unstable manifolds that are locally homeomorphic to the stable and unstable subspaces.
\begin{proof}
Recall the proof of Theorem \ref{IF:construction}, the open dense subset $G$ is the following
\[
G=(\BR^d-F)\cap G_1\cap G_2
\]
where $G_1$ is the set of elements $\vec{x}$ such that $f(L_{\vec{x}})$ is constant in a neighborhood of $\vec{x}$ and $G_2$ is the set of elemnts $\vec{x}$ such that $f(l_{\vec{x}})$ is constant in a neighborhood of $\vec{x}$. In the case when the critical points are at most countable and all saddle points are strict, by the Center-stable Manifold Theorem, there exist stable and unstable manifolds whose dimensions are strictly less than $d$ corresponding to the strict saddle points. Then in such case, denote $W^s_{\vec{x}^*}$ and $W^u_{\vec{x}^*}$ the stable and unstable manifolds of the saddle point $\vec{x}^*$ respectively, we have
\[
G_1=G_2=\BR^d-\bigcup_{\vec{x}^*}\left(W^s_{\vec{x}^*}\cup W^u_{\vec{x}^*}\right)
\]
where the set of $\vec{x}^*$ is at most countable. Then the set $G$ is nothing but the following
\[
G=\mathbb{R}^d-F-\bigcup_{\vec{x}^*}\left(W^s_{\vec{x}^*}\cup W^u_{\vec{x}^*}\right).
\]
Note that the set of cirtical points and countable union of stable and unstable manifolds are all of measure zero, thus the complement of $G$ is of measure zero.
\end{proof}

\noindent
\textbf{Proof of Theorem \ref{IF:compact}}
We give the detailed proof by modifying the proof of Theorem \ref{invariant functions:GD} and \ref{IF:construction} into two parts: 
\begin{itemize}
\item \textbf{Existence of many invariant functions}
\item \textbf{Representation of invariant functions}
\end{itemize}
\begin{proof}
\textbf{Existence of many invariant functions} 
\begin{lemma}\label{fundlemma1x}
 For each $\vec{x}_0\in M-F$, there exists a neighborhood $N$ of $\vec{x}_0$, such that $T^r(N)\cap T^s(N)=\emptyset$ if $r\ne s$.
\end{lemma}
\begin{proof}
1. For each $x_0\in M-F$, there exists $N$ such that $T^{-1}(N)\cap N=\emptyset$.
\\
Since $f(T^{-1}(\vec{x}_0))-f(\vec{x}_0)=\Delta>0$ (clearly it is from that $f(T^{-1}(\vec{x}))>f(TT^{-1}(\vec{x}))$), $f$ is continuous, so there exists neighborhood $U$ of $\vec{x}_0$ such that $f(\vec{x})<f(\vec{x}_0)+\frac{\Delta}{3}$ for all $\vec{x}\in U$, and neighborhood $V$ of $T^{-1}(\vec{x}_0)$ such that $f(\vec{y})>f(T^{-1}(\vec{x}_0))-\frac{\Delta}{3}$ for all $\vec{y}\in V$. We have that $T(V)\cap U$ is open since $T^{-1}$ is homeomorphism. Let $N$ be a neighborhood of $\vec{x}$ in $T(V)\cap U$, then $N\subset U$ and $T^{-1}(N)\subset V$. We have 
\[
f(\vec{x})<f(\vec{x}_0)+\frac{\Delta}{3}<f(T^{-1}(\vec{x}_0))-\frac{\Delta}{3}<f(\vec{y}) \ \ \ \text{for all}\ \ \ \vec{x}\in N, \ \ \vec{y}\in V.
\]
So $T^{-1}(N)\cap N=\emptyset$.
\\
\\
2. $T^{-m}(N)\cap N=\emptyset$ for all $m\ge 1$.
\\
Suppose $\vec{x}\in N$ and $\vec{z}\in T^{-m}(N)$, then $\vec{z}=T^{-m}(\vec{u})$ for some $\vec{u}\in N$, and then 
\[
f(\vec{z})=f(T^{-m}(\vec{u}))>f(T^{-1}(\vec{u}))>f(\vec{x}),
\]
since $T^{-1}(\vec{u})\in T^{-1}(N)$. This shows that $\vec{z}\ne \vec{x}$.
\\
\\
3. $T^{-r}(N)\cap T^{-s}(N)=\emptyset$ if $r\ne s$.
\\
Assume $r>s$, let $\vec{y}\in T^{-r}(N)\cap T^{-s}(N)$. Then $T^{r}(\vec{y})\in N$ and $T^{-r+s}(T^{r}(\vec{y}))=T^{s}(\vec{y})\in N$. So $T^{r}(\vec{y})\in N\cap T^{-r+s}(N)$ which is impossible from 2. 
\end{proof}
The following lemma is from \cite{Stebe}.
\begin{lemma}[Lemma 2.3, \cite{Stebe}]
Let $T$ be a homeomorphism of $M$ onto itself. There is a positive number $\epsilon$ such that if $\vec{x}$ is a point of $M$ but not a fixed point of $T$, there is at least one element of the sequence $\{T^n(\vec{x})\}$ at distance greater than or equal to $\epsilon$ from the set of fixed points of $T$.
\end{lemma}

\begin{lemma}\label{fundlemma3x}
Let $T:M\rightarrow M$ be a homeomorphism, then $T$ has a fundamental set.
\end{lemma}
\begin{proof}
Let $D_{\epsilon}=\{\vec{x}\in M:dist(x,F)\ge\epsilon\}$, where $\epsilon\ge 0$ may be chosen so that $D_{\epsilon}\ne\emptyset$. Since $D_{\epsilon}\cap F=\emptyset$, from Lemma \ref{fundlemma1}, for each $\vec{x}\in D_{\epsilon}$, there exists a neighborhood $N_{\vec{x}}$ of $\vec{x}$ such that $T^n(N_{\vec{x}})$ are disjoint. Since $M$ is a second countable space, $D_{\epsilon}$ has countable basis, i.e. it has countable open cover $N_1,...,N_r,...$. Then each sequence $T^n(\vec{x})$ for $\vec{x}\in M-F$, $T^n(\vec{x})$ only meets $N_i$ once.
\\
Let 
\begin{align}
L_1&=N_1
\\
L_2&=N_2-\cup_{-\infty}^{\infty}T^n(N_1)
\\
L_3&=N_3-\cup_{-\infty}^{\infty}T^n(N_1)-\cup_{-\infty}^{\infty}T^n(N_2)
\\
&\vdots
\\
L_r&=N_r-\cup_{-\infty}^{\infty}T^n(N_1)-...-\cup_{-\infty}^{\infty}T^n(N_{r-1}),
\\
&\vdots
\end{align}
then $S=\cup_{i=1}^{\infty} L_i$ is a fundamental set.
\end{proof}

\begin{lemma}\label{fundlemma4x}
Let $T$ be the gradient descent for $f:M\rightarrow\mathbb{R}$, and $F$ be the set of fixed points of $T$, there exist d $T$-invariant functions of $T$ which are continuous and independent on an open dense subset of $M-F$. 
\end{lemma}

\begin{proof}
Let $S\subset M$ be a fundamental set for $T$. Let $\partial S$ be the boundary of $S$ and let $B=\cup_{-\infty}^{\infty}T^n(\partial S)$. Then $M-F-B$ is dense in $M-F$.
Denote $M'=M-F-B$ and recall that
\[
[\vec{x}]=\{T^n(\vec{x}):n\in\mathbb{Z}\}
\]
and 
\[
M'/T:=\{[\vec{x}]:\vec{x}\in M'\}.
\]
Let $\varphi(\vec{x})$ be the element of $\{T^n(\vec{x})\}$ in $S$. We next show that it is continuous. 
\\
If $\vec{x}\in M'$, $\varphi(\vec{x})$ is the unique intersection of $\{T^n(\vec{x})\}$ with $S$. Hence there is an integer $m$ such that $T^m(\vec{x})\in S$. Since $\vec{x}\in B$, $T^m(\vec{x})$ is an interior point of $S$. Let $U$ be a neighborhood of $T^m(\vec{x})$ in $S$. Since $T^m$ is continuous, $V=(T^m)^{-1}(U)=T^{-m}(U)$ is a neighborhood of $\vec{x}$. If $\vec{y}\in V$, $T^m(\vec{y})\in S$ so that $\varphi(\vec{y})=T^m(\vec{y})$ for all $\vec{y}\in V$. Hence $\varphi$ is continuous in a neighborhood of $\vec{x}\in M'$, $M'$ is open and $\varphi=T^m$ for some $m$ in a neighborhood of $\vec{x}\in M$.
\\
We set 
\[
\varphi(\vec{x})=(\varphi_1(\vec{x}),...,\varphi_d(\vec{x}))
\] 
so that the $\varphi_i(\vec{x})$ are the components of $\varphi(\vec{x})$, it follows that the $\varphi_i(\vec{x})$ are continuous and independent on $M'$, since $\varphi(\vec{x})$ is a local homeomorphism on $M'$. Since $\varphi(T(\vec{x}))=\varphi(\vec{x})$, $\varphi_i(T(\vec{x}))=\varphi_i(\vec{x})$, meaning that $\varphi_i(\vec{x})$ are $T$-invariant, and then the proof is complete.
\end{proof}
The existence of $d$ invariant functions relies on the property that $f(T_{\eta}(\vec{x}))<f(\vec{x})$. For each $\vec{x}\in M$, there exists a neighborhood $U_{\vec{x}}$ of $\vec{x}$ and stepsize $\eta_{\vec{x}}$ such that $f(T_{\eta_{\vec{x}}}(\vec{y}))<f(\vec{y})$ for all $\vec{y}\in U_{\vec{x}}$. Since $M$ is compact, one chooses $\eta$ from the finite subcovering of $\bigcup U_{\vec{x}}$ such that $f(T_{\eta}(\vec{x}))<f(\vec{x})$ holds for all $\vec{x}\in M$. On the other hand, since $T_{\eta}\rightarrow Id$ as $\eta\rightarrow 0$, the determinant of the Jacobian of $T_{\eta}$, denoted as $|J_{\eta}|$ is a continuous function with respect to $\eta$. $|J_{\eta}|\rightarrow 1$ as $\eta\rightarrow 0$, so at each point $\vec{x}$, one can choose $\eta$ small enough so that $T_{\eta}$ is a local diffeomorphism. By compactness of $M$, $\eta$ can be chosen such that for all $\vec{x}\in M$, $T_{\eta}$ is a local diffeomorphism. Since $M$ is simply-connected and compact Hausdorff space, so the pre-image of a compact set under $T$ is always compact since any open cover of the pre-image can be extended to an open cover of $M$, and from the compactness of $M$, one can choose a finite sub-cover of $M$, so a sub-cover of this pre-image. This means $T$ is a proper map. Thus theorem \ref{proper map} implies that $T_{\eta}$ is a global diffeomorphism, i.e. $T^{-1}$ is well defined on all over $M$.
\end{proof}
Now we have finished the proof of existence of $d$ invariant functions and we proceed to prove the representation of the invariant functions.

\begin{proof}
\textbf{Representation of invariant functions}
\begin{lemma}
The set of cluster points of $\{T^n(\vec{x})\}_{n\in\mathbb{Z}}$ is the union of $L_{\vec{x}}$ and $l_{\vec{x}}$. The value of $f$ is constant on each of $L_{\vec{x}}$ and $l_{\vec{x}}$. If $f(L_{\vec{x}})$ denotes the value of $f$ on $l_{\vec{x}}$ and $f(L_{\vec{x}})$ denotes the value of $f$ on $L_{\vec{x}}$ we have $f(L_{\vec{x}})>f(l_{\vec{x}})$ whenever $\vec{x}$ is not a fixed point of $T$ in $M$.
\end{lemma}

\begin{proof}
Let $\vec{a},\vec{b}\in l_{\vec{x}}$, then by definition of cluster point, we have two subsequences $\{n_i\}_{\mathbb{Z}_+}$ and $\{n_j\}_{\mathbb{Z}_+}$ of $n\ge 0$, such that
\[
\lim_{i\rightarrow\infty}\norm{T^{n_i}(\vec{x})-\vec{a}}=0
\]
and
\[
\lim_{j\rightarrow\infty}\norm{T^{n_j}(\vec{x})-\vec{b}}=0.
\]
Since $f$ is continuous, we have that 
\[
\lim_{i\rightarrow\infty}f(T^{n_i}(\vec{x}))=f(\vec{a})
\]
and 
\[
\lim_{j\rightarrow\infty}f(T^{n_j}(\vec{x}))=f(\vec{b}).
\]
By the fact that $\lim_{i\rightarrow\infty}f(T^{n_i}(\vec{x}))=\lim_{i\rightarrow\infty}f(T^{n_i}(\vec{x}))=\text{local minimum with initial condition}\ \ \vec{x}$, we conclude $f(\vec{a})=f(\vec{b})$.
\\
The other case, let $\{n\}=\{n=-1,-2...\}$ be the sequence of negative integers and $\vec{c},\vec{d}\in L_{\vec{x}}$, there exist two subsequences negative integers $\{n_i\}_{i\in\mathbb{Z}_+}$ and $\{n_j\}_{j\in\mathbb{Z}_+}$ of $\{n\}$, such that
\[
\lim_{i\rightarrow\infty}\norm{T^{n_i}(\vec{x})-\vec{c}}=0
\]
and
\[
\lim_{j\rightarrow\infty}\norm{T^{n_j}(\vec{x})-\vec{d}}=0,
\]
and by the continuity of $f$, we have
\[
\lim_{i\rightarrow\infty}f(T^{n_i}(\vec{x}))=f(\vec{c})
\]
and 
\[
\lim_{j\rightarrow\infty}f(T^{n_j}(\vec{x}))=f(\vec{d}).
\]
Since $\lim_{i\rightarrow\infty}f(T^{n_i}(\vec{x}))=\lim_{j\rightarrow\infty}f(T^{n_j}(\vec{x}))=\text{local maximum with initial condition}\ \ \vec{x}$, $f(\vec{c})=f(\vec{d})$.
\end{proof}

\begin{lemma}\label{B2}
Let $\vec{x}_0$ be an element of $M$. Either there is a neighborhood $N$ of $\vec{x}_0$ such that $f(L_{\vec{x}})=f(L_{\vec{x}_0})$ for all $\vec{x}\in N$ or in every neighborhood of $\vec{x}_0$ there is an $\vec{x}$ such that $f(L_{\vec{x}})>f(L_{\vec{x}_0})$. 
\end{lemma}

\begin{proof}
Suppose there is a neighborhood $N_1$ of $\vec{x}_0$ in $M$ such that $f(L_{\vec{x}_0})\ge f(L_{\vec{x}})$ for all $x\in N_1$. Let $\xi$ be a positive number. Let $S_{\xi}=\{\vec{x}:f(L_{\vec{x}})\ge f(L_{\vec{x}_0})-\xi\}$. We show that $S_{\xi}$ is open. If $\vec{x}$ is an element of $S_{\xi}$, there is an $m$ such that $f(T^m(\vec{x}))>f(L_{\vec{x}_0})-\xi$. Since $T^m$ is continuous, there is a neighborhood $N_{\vec{x}}$ of $\vec{x}$ such that $f(T^m(\vec{y}))>f(L_{\vec{x}_0})-\xi$ for all $\vec{y}$ in $N_{\vec{x}}$. But $f(L_{\vec{y}})\ge f(T^m(\vec{y}))$ for all $\vec{y}\in M$ so that $f(L_{\vec{y}})\in S_{\xi}$ for all $y\in N_{\vec{x}}$. Hence $S_{\xi}$ is open. Let $N(\xi)=S_{\xi}\cap N_{\vec{x}_0}$. Since $\vec{x}_0$ is an element of $S_{\xi}$ for all positive $\xi$, $N(\xi)$ is not empty for $\xi>0$. Since $N(\xi)$ is contained in $N_{\vec{x}_0}$ and $S_{\xi}$, $f(L_{\vec{x}_0})\ge f(L_{\vec{x}})\ge f(L_{\vec{x}_0})-\xi$ for all $\vec{x}$ in $N(\xi)$. Since the points of $L_{\vec{x}}$ are in $F$, the set of fixed points of $T$, $f(L_{\vec{x}})$ can assume only finitely many values. Hence for $\xi$ sufficiently small
\[
f(L_{\vec{x}_0})\ge f(L_{\vec{x}})\ge f(L_{\vec{x}_0})-\xi
\]
implies that $f(L_{\vec{x}})=f(L_{\vec{x}_0})$, and so for some $\xi$, $\vec{x}\in N(\xi)$ implies that $f(L_{\vec{x}})=f(L_{\vec{x}_0})$.
\end{proof}

\begin{lemma}
Let $\vec{x}_0$ be an element of $M$. Either there is a neighborhood $N_{\vec{x}_0}$ of $\vec{x}_0$ in $M$ such that $f(L_{\vec{x}_0})=f(L_{\vec{x}})$ for all $\vec{x}$ in $N_{\vec{x}_0}$ or every neighborhood $N$ of $\vec{x}_0$ contains an open subset $V_{N}$ such that $f(L_{\vec{y}})=f(L_{\vec{z}})$ for all $\vec{y}$ and $\vec{z}$ in $V_{N}$.
\end{lemma}

\begin{proof}
Suppose $\vec{x}_0$ is an element of $M$ and there is no neighborhood $U$ of $\vec{x}_0$ in $M$ such that $f(L_{\vec{x}})=f(L_{\vec{x}_0})$ for all $\vec{x}$ in $U$. Let $N$ be a neighborhood of $\vec{x}_0$. According to the lemma \ref{B2}, there is an element $\vec{x}$ of $N$ such that $f(L_{\vec{x}})>f(L_{\vec{x}_0})$. Let $K$ be the least upper bound of $f(L_{\vec{x}})$ for $\vec{x}$ in $N$. Since the range of $f(L_{\vec{x}})$ is finite, there is a point $\vec{y}$ of $N$ such that $f(L_{\vec{y}})=K$. Thus $f(L_{\vec{y}})\ge f(L_{\vec{x}})$ for all $\vec{x}$ in $N$, and $N$ is a neighborhood of $\vec{y}$. By lemma \ref{B2}, there is a neighborhood $U$ of $\vec{y}$ such that $f(L_{\vec{y}})=f(L_{\vec{x}})$ for all $\vec{x}\in U$. Let $V_{N}=N\cap U$.
\end{proof}

\begin{lemma}
Let $\vec{x}_0$ be an element of $M$. Either there is a neighborhood $N_{\vec{x}_0}$ of $\vec{x}_0$ in $M$ such that $f(l_{\vec{x}_0})=f(l_{\vec{x}})$ for all $\vec{x}$ in $N_{\vec{x}_0}$, or every neighborhood $N$ of $\vec{x}_0$ contains an open subset $U_{N}$ such that $f(l_{\vec{y}})=f(l_{\vec{z}})$ for all $\vec{y}$ and $\vec{z}$ in $U_N$.
\end{lemma}

\begin{proof}
Using the fact that if $T$ is a homeomorphism of $M$ onto itself, $T^{-1}$ is defined and either $x=T^{-1}(x)$ or $f(T^{-1}(\vec{x}))<f(\vec{x})$, we can modify the above arguments by replacing $T$ with $T^{-1}$ and reversing the inequalities to have the results about the function $f(l_{\vec{x}})$. Suppose there is a neighborhood $N_1$ of $\vec{x}_0$ in $M$ such that $f(l_{\vec{x}_0})\le f(l_{\vec{x}})$ for all $\vec{x}\in N_1$. Let $\xi$ be a positive number. Let $S_{\xi}=\{\vec{x}:f(l_{\vec{x}})\le f(l_{\vec{x}_0})+\xi\}$. We show that $S_{\xi}$ is open. If $\vec{x}$ is an element of $S_{\xi}$, there is an $m$ such that $f(T^m(\vec{x}))<f(l_{\vec{x}_0})+\xi$. Since $T^m$ is continuous, there is a neighborhood $N_{\vec{x}}$ of $\vec{x}$ such that $f(T^m(\vec{y}))<f(l_{\vec{x}_0})+\xi$ for all $\vec{y}$ in $N_{\vec{x}}$. But $f(l_{\vec{y}})\le f(T^m(\vec{y}))$ for all $\vec{y}\in M$ so that $f(l_{\vec{y}})\in S_{\xi}$ for all $\vec{y}\in N_{\vec{x}}$. Hence $S_{\xi}$ is open. Let $N(\xi)=S_{\xi}\cap N_{\vec{x}_0}$. Since $\vec{x}_0$ is an element of $S_{\xi}$ for all positive $\xi$, $N(\xi)$ is not empty for $\xi>0$. Since $N(\xi)$ is contained in $N_{\vec{x}_0}$ and $S_{\xi}$, $f(l_{\vec{x}_0})\le f(l_{\vec{x}})\le f(l_{\vec{x}_0})+\xi$ for all $\vec{x}$ in $N(\xi)$. Since the points of $l_{\vec{x}}$ are in the fixed point set $F$, $f(l_{\vec{x}})$ can assume only finitely many values. Hence for $\xi$ sufficiently small 
\[
f(l_{\vec{x}_0})\le f(l_{\vec{x}})\le f(l_{\vec{x}_0})+\xi
\]
implies that $f(l_{\vec{x}})=f(l_{\vec{x}_0})$, and so for some $\xi$, $\vec{x}\in N(\xi)$ implies that $f(l_{\vec{x}})=f(l_{\vec{x}_0})$. Next, suppose $\vec{x}_0$ is an element of $M$ and there is no neighborhood $U$ of $\vec{x}_0$ in $M$ such that $f(l_{\vec{x}})=f(l_{\vec{x}_0})$ for all $\vec{x}$ in $U$. Let $N$ be a neighborhood of $\vec{x}_0$. According to the above arguments, there is an element $\vec{x}$ of $N$ such that $f(l_{\vec{x}})<f(l_{\vec{x}_0})$. Let $K$ be the least upper bound of $f(L_{\vec{x}})$ for $\vec{x}$ in $N$. Since the range of $f(l_{\vec{x}})$ is finite, there is a point $\vec{y}$ of $N$ such that $f(l_{\vec{y}})=K$. Thus $f(l_{\vec{y}})\le f(l_{\vec{x}})$ for all $\vec{x}$ in $N$, and $N$ is a neighborhood of $\vec{y}$. Thus there is a neighborhood $U$ of $\vec{y}$ such that $f(l_{\vec{y}})=f(l_{\vec{x}})$ for all $\vec{x}\in U$. Let $U_N=N\cap U$, the proof completes.
\end{proof}

Next we complete the proof of theorem. For each $\vec{x}\in M-F$, $\Phi(\vec{x})$ is convergent. Let $G_1$ be the set of all elements $\vec{x}$ of $M$ such that $f(L_{\vec{x}})$ is constant in a neighborhood of $\vec{x}$. Let $G_2$ be the set of all elements $\vec{x}$ of $M$ such that $f(l_{\vec{x}})$ is a constant in a neighborhood of $\vec{x}$. Notice that $G=(M-F)\cap G_1\cap G_2$ is an open dense subset of $M-F$. For each $\vec{x}\in M$, let 
\[
S(\vec{x})=\sum_{n=-\infty}^{\infty}f(T^{n-1}(\vec{x}))-f(T^n(\vec{x})).
\]
Clearly $S(\vec{x})$ converges at each $\vec{x}$ to $f(L_{\vec{x}})-f(l_{\vec{x}})$. Let $\vec{y}$ be an element of $G$. There is a neighborhood $U$ of $\vec{y}$ such that $S(\vec{x})$ represents the constant function in $U$. Since $\vec{y}\notin F$ and $F$ is compact, there is a neighborhood $W$ of $\vec{y}$ such that $\bar{W}\subset U\cap V$. Then $S(\vec{x})$ is a series of positive terms converging to a continuous function on $\bar{W}$ and $S(\vec{x})$ converges uniformly on $\bar{W}$. Let $p(\vec{x})$ be any bounded function continuous on $M$. The series
\[
\Phi(\vec{x})=\sum_{n=-\infty}^{\infty}p(T^n(\vec{x}))\left(f(T^{n-1}(\vec{x}))-f(T^n(\vec{x}))\right)
\]
converges uniformly on $\bar{W}$ since $f$ is taken to be bounded on $M$. Since $p$, $f$ and $T$ are continuous , $\Phi(\vec{x})$ is continuous on $\bar{W}$ and hence at $\vec{y}$. The invariance is obvious, so the proof completes.
\end{proof}

\noindent
\textbf{Proof of Theorem \ref{IF:manifold}}
In this proof, we use unbold face $x$ to distinguish from $\vec{x}$ that represents a point in Euclidean space.
\begin{proof}
\textbf{Existence of invariant functions}
\\
The following descent lemma for Riemannian gradient descent \cite{CB} is analogy of classic gradient descent in Euclidean space. We include its proof for completeness. 
\begin{lemma}[Lemma C.2 \cite{CB}]
Let $\nabla\hat{f}_x$ be $l$-Lipschitz continuous along the line segment connecting $s_j$ to $s_{j+1}$, related by $s_{j+1}=s_j-\alpha\eta\nabla\hat{f}_x(s_j)$ with $\eta=1/l$ and $\alpha\in[0,1]$. Then,
\[
\hat{f}_x(s_{j+1})-\hat{f}_x(s_j)\le-\frac{\alpha\eta}{2}\norm{\nabla\hat{f}_x(s_j)}^2
\]
Moreover, $f(\Retr_x(s))\le f(x)$.
\end{lemma}

A direct consequence of above lemma is for manifold gradient descent $T(x_k)=\Retr_x(-\eta\grad f(x_k))$, we have $f(T(x_k))\le f(x_k)$.

\begin{lemma}\label{fundlemma1xx}
 For each $x_0\in\M-F$, there exists a neighborhood $N$ of $x_0$, such that $T^r(N)\cap T^s(N)=\emptyset$ if $r\ne s$.
\end{lemma}
\begin{proof}
The proof is completed in three steps.
\\
1. For each $x_0\in \M-F$, there exists $N$ such that $T^{-1}(N)\cap N=\emptyset$.
\\
Since $f(T^{-1}(x_0))-f(x_0)=\Delta>0$ (clearly it is from that $f(T^{-1}(x))>f(TT^{-1}(x))$), $f$ is continuous, so there exists neighborhood $U$ of $\vec{x}_0$ such that $f(x)<f(x_0)+\frac{\Delta}{3}$ for all $x\in U$, and neighborhood $V$ of $T^{-1}(x_0)$ such that $f(y)>f(T^{-1}(x_0))-\frac{\Delta}{3}$ for all $y\in V$. We have that $T(V)\cap U$ is open since $T^{-1}$ is homeomorphism. Let $N$ be a neighborhood of $x$ in $T(V)\cap U$, then $N\subset U$ and $T^{-1}(N)\subset V$. We have 
\[
f(x)<f(x_0)+\frac{\Delta}{3}<f(T^{-1}(x_0))-\frac{\Delta}{3}<f(y) \ \ \ \text{for all}\ \ \ x\in N, \ \ y\in V.
\]
So $T^{-1}(N)\cap N=\emptyset$.
\\
\\
2. $T^{-m}(N)\cap N=\emptyset$ for all $m\ge 1$.
\\
Suppose $x\in N$ and $z\in T^{-m}(N)$, then $z=T^{-m}(u)$ for some $\vec{u}\in N$, and then 
\[
f(z)=f(T^{-m}(u))>f(T^{-1}(u))>f(x),
\]
since $T^{-1}(u)\in T^{-1}(N)$. This shows that $z\ne x$.
\\
\\
3. $T^{-r}(N)\cap T^{-s}(N)=\emptyset$ if $r\ne s$.
\\
Assume $r>s$, let $y\in T^{-r}(N)\cap T^{-s}(N)$. Then $T^{r}(y)\in N$ and $T^{-r+s}(T^{r}(y))=T^{s}(y)\in N$. So $T^{r}(y)\in N\cap T^{-r+s}(N)$ which is impossible from 2.
\end{proof}

\begin{lemma}\label{fundlemma3xx}
Let $T:\M\rightarrow \M$ be a homeomorphism, then $T$ has a fundamental set.
\end{lemma}

\begin{proof}
By Lemma 2.3 \cite{Stebe}, let $D_{\epsilon}=\{x\in \M:dist(x,F)\ge\epsilon\}$, this set is not empty, where $\epsilon\ge 0$ may be chosen so that $D_{\epsilon}\ne\emptyset$. Since $D_{\epsilon}\cap F=\emptyset$, from Lemma \ref{fundlemma1xx}, for each $x\in D_{\epsilon}$, there exists a neighborhood $N_{x}$ of $x$ such that $T^n(N_{x})$ are disjoint. Since $\M$ is a second countable space, $D_{\epsilon}$ has countable basis, i.e. it has countable open cover $N_1,...,N_r,...$. Then each sequence $T^n(x)$ for $x\in \M-F$, $T^n(x)$ only meets $N_i$ once.
\\
Let 
\begin{align}
L_1&=N_1
\\
L_2&=N_2-\cup_{-\infty}^{\infty}T^n(N_1)
\\
L_3&=N_3-\cup_{-\infty}^{\infty}T^n(N_1)-\cup_{-\infty}^{\infty}T^n(N_2)
\\
&\vdots
\\
L_r&=N_r-\cup_{-\infty}^{\infty}T^n(N_1)-...-\cup_{-\infty}^{\infty}T^n(N_{r-1}),
\\
&\vdots
\end{align}
then $S=\cup_{i=1}^{\infty} L_i$ is a fundamental set.
\end{proof}

\begin{lemma}\label{fundlemma4xx}
Let $T$ be the manifold gradient descent for $f:\M\rightarrow\mathbb{R}$, and $F$ be the set of fixed points of $T$, there exist d $T$-invariant functions of $T$ which are continuous and independent on an open dense subset of $\M-F$. 
\end{lemma}

\begin{proof}
Let $S\subset\M$ be a fundamental set for $T$. Let $\partial S$ be the boundary of $S$ and let $B=\cup_{-\infty}^{\infty}T^n(\partial S)$. Then $\M-F-B$ is dense in $\M-F$.
Denote $M'=\M-F-B$ and recall that
\[
[x]=\{T^n(x):n\in\mathbb{Z}\}
\]
and 
\[
M'/T:=\{[x]:x\in M'\}.
\]
Let $\varphi(x)$ be the element of $\{T^n(x)\}$ in $S$. We next show that it is continuous. 
\\
If $x\in M'$, $\varphi(x)$ is the unique intersection of $\{T^n(x)\}$ with $S$. Hence there is an integer $m$ such that $T^m(x)\in S$. Since $x\in B$, $T^m(x)$ is an interior point of $S$. Let $U$ be a neighborhood of $T^m(x)$ in $S$. Since $T^m$ is continuous, $V=(T^m)^{-1}(U)=T^{-m}(U)$ is a neighborhood of $x$. If $\vec{y}\in V$, $T^m(y)\in S$ so that $\varphi(y)=T^m(y)$ for all $y\in V$. Hence $\varphi$ is continuous in a neighborhood of $x\in M'$, $M'$ is open and $\varphi=T^m$ for some $m$ in a neighborhood of $x\in \M$.
\\
We set 
\[
\varphi(x)=(\varphi_1(x),...,\varphi_d(x))
\] 
so that the $\varphi_i(x)$ are the components of $\varphi(x)$, it follows that the $\varphi_i(x)$ are continuous and independent on $M'$, since $\varphi(x)$ is a local homeomorphism on $M'$. Since $\varphi(T(x))=\varphi(x)$, $\varphi_i(T(x))=\varphi_i(x)$, meaning that $\varphi_i(x)$ are $T$-invariant, and then the proof is complete.
\end{proof}
We complete the proof of the existence of invariant function by showing that by choosing small enough stepsize $\eta$, the manifold gradient descent is a global diffeomorphism on a simply-connected manifold $\M$. Under a proper choice of local coordinate system, the Jacobian of the differential of gradient descent algorithm on a manifold can be written as $I-\eta\nabla^2f(x)$. Since the determinant of the Jacobian is a continuous function of its coefficients, by taking $\eta$ small enough, the determinant is close to 1 so bounded away from 0, which implies $T$ is a local diffeomorphism. By assumption in the theorem, $\M$ is simply connected, so a proper local diffeomorphism is a global diffeomorphism and then the inverse $T^{-1}$ is well defined. Then the existence of $d$ $T$-invariant functions follows from Lemma \ref{fundlemma4xx}.
\end{proof}

\begin{proof}
\textbf{Representation of invariant functions}
\begin{lemma}
The set of cluster points of $\{T^n(x)\}_{n\in\mathbb{Z}}$ is the union of $L_{x}$ and $l_{x}$. The value of $f$ is constant on each of $L_{x}$ and $l_{x}$. If $f(L_{x})$ denotes the value of $f$ on $l_{x}$ and $f(L_{x})$ denotes the value of $f$ on $L_{x}$ we have $f(L_{x})>f(l_{x})$ whenever $x$ is not a fixed point of $T$ in $M$.
\end{lemma}

\begin{proof}
Denote $d(\cdot,\cdot)$ the geodesic distance on $\M$. Let $a,b\in l_{x}$, then by definition of cluster point, we have two subsequences $\{n_i\}_{\mathbb{Z}_+}$ and $\{n_j\}_{\mathbb{Z}_+}$ of $n\ge 0$, such that
\[
\lim_{i\rightarrow\infty}d(T^{n_i}(x)-a)=0
\]
and
\[
\lim_{j\rightarrow\infty}d(T^{n_j}(x)-b)=0.
\]
Since $f$ is continuous, we have that 
\[
\lim_{i\rightarrow\infty}f(T^{n_i}(x))=f(a)
\]
and 
\[
\lim_{j\rightarrow\infty}f(T^{n_j}(x))=f(b).
\]
By the fact that $\lim_{i\rightarrow\infty}f(T^{n_i}(x))=\lim_{i\rightarrow\infty}f(T^{n_i}(x))=\text{local minimum with initial condition}\ \ x$, we conclude $f(a)=f(b)$.
\\
The other case, let $\{n\}=\{n=-1,-2...\}$ be the sequence of negative integers and $c,d\in L_{x}$, there exist two subsequences negative integers $\{n_i\}_{i\in\mathbb{Z}_+}$ and $\{n_j\}_{j\in\mathbb{Z}_+}$ of $\{n\}$, such that
\[
\lim_{i\rightarrow\infty}d(T^{n_i}(x)-c)=0
\]
and
\[
\lim_{j\rightarrow\infty}d(T^{n_j}(x)-d)=0,
\]
and by the continuity of $f$, we have
\[
\lim_{i\rightarrow\infty}f(T^{n_i}(x))=f(c)
\]
and 
\[
\lim_{j\rightarrow\infty}f(T^{n_j}(x))=f(d).
\]
Since $\lim_{i\rightarrow\infty}f(T^{n_i}(x))=\lim_{j\rightarrow\infty}f(T^{n_j}(x))=\text{local maximum with initial condition}\ \ x$, $f(c)=f(d)$.
\end{proof}

\begin{lemma}\label{B2xx}
Let $x_0$ be an element of $M$. Either there is a neighborhood $N$ of $x_0$ such that $f(L_{x})=f(L_{x_0})$ for all $x\in N$ or in every neighborhood of $x_0$ there is an $\vec{x}$ such that $f(L_{x})>f(L_{x_0})$. 
\end{lemma}

\begin{proof}
Suppose there is a neighborhood $N_1$ of $x_0$ in $M$ such that $f(L_{x_0})\ge f(L_{x})$ for all $x\in N_1$. Let $\xi$ be a positive number. Let $S_{\xi}=\{x:f(L_{x})\ge f(L_{x_0})-\xi\}$. We show that $S_{\xi}$ is open. If $x$ is an element of $S_{\xi}$, there is an $m$ such that $f(T^m(x))>f(L_{x_0})-\xi$. Since $T^m$ is continuous, there is a neighborhood $N_{x}$ of $x$ such that $f(T^m(y))>f(L_{x_0})-\xi$ for all $y$ in $N_{x}$. But $f(L_{y})\ge f(T^m(y))$ for all $y\in M$ so that $f(L_{y})\in S_{\xi}$ for all $y\in N_{x}$. Hence $S_{\xi}$ is open. Let $N(\xi)=S_{\xi}\cap N_{x_0}$. Since $x_0$ is an element of $S_{\xi}$ for all positive $\xi$, $N(\xi)$ is not empty for $\xi>0$. Since $N(\xi)$ is contained in $N_{x_0}$ and $S_{\xi}$, $f(L_{x_0})\ge f(L_{x})\ge f(L_{x_0})-\xi$ for all $x$ in $N(\xi)$. Since the points of $L_{x}$ are in $F$, the set of fixed points of $T$, $f(L_{x})$ can assume only finitely many values. Hence for $\xi$ sufficiently small
\[
f(L_{x_0})\ge f(L_{x})\ge f(L_{x_0})-\xi
\]
implies that $f(L_{x})=f(L_{x_0})$, and so for some $\xi$, $x\in N(\xi)$ implies that $f(L_{x})=f(L_{x_0})$.
\end{proof}

\begin{lemma}
Let $x_0$ be an element of $M$. Either there is a neighborhood $N_{x_0}$ of $x_0$ in $M$ such that $f(L_{x_0})=f(L_{x})$ for all $x$ in $N_{x_0}$ or every neighborhood $N$ of $x_0$ contains an open subset $V_{N}$ such that $f(L_{y})=f(L_{z})$ for all $y$ and $z$ in $V_{N}$.
\end{lemma}

\begin{proof}
Suppose $x_0$ is an element of $M$ and there is no neighborhood $U$ of $x_0$ in $M$ such that $f(L_{x})=f(L_{x_0})$ for all $x$ in $U$. Let $N$ be a neighborhood of $x_0$. According to the lemma \ref{B2xx}, there is an element $x$ of $N$ such that $f(L_{x})>f(L_{x_0})$. Let $K$ be the least upper bound of $f(L_{x})$ for $x$ in $N$. Since the range of $f(L_{x})$ is finite, there is a point $y$ of $N$ such that $f(L_{y})=K$. Thus $f(L_{y})\ge f(L_{x})$ for all $x$ in $N$, and $N$ is a neighborhood of $y$. By lemma \ref{B2}, there is a neighborhood $U$ of $y$ such that $f(L_{y})=f(L_{x})$ for all $x\in U$. Let $V_{N}=N\cap U$.
\end{proof}

\begin{lemma}
Let $x_0$ be an element of $M$. Either there is a neighborhood $N_{x_0}$ of $x_0$ in $M$ such that $f(l_{x_0})=f(l_{x})$ for all $x$ in $N_{x_0}$, or every neighborhood $N$ of $x_0$ contains an open subset $U_{N}$ such that $f(l_{y})=f(l_{z})$ for all $y$ and $z$ in $U_N$.
\end{lemma}

\begin{proof}
Using the fact that if $T$ is a homeomorphism of $M$ onto itself, $T^{-1}$ is defined and either $x=T^{-1}(x)$ or $f(T^{-1}(x))<f(x)$, we can modify the above arguments by replacing $T$ with $T^{-1}$ and reversing the inequalities to have the results about the function $f(l_{x})$. Suppose there is a neighborhood $N_1$ of $x_0$ in $M$ such that $f(l_{x_0})\le f(l_{x})$ for all $x\in N_1$. Let $\xi$ be a positive number. Let $S_{\xi}=\{x:f(l_{x})\le f(l_{x_0})+\xi\}$. We show that $S_{\xi}$ is open. If $x$ is an element of $S_{\xi}$, there is an $m$ such that $f(T^m(x))<f(l_{x_0})+\xi$. Since $T^m$ is continuous, there is a neighborhood $N_{x}$ of $x$ such that $f(T^m(y))<f(l_{x_0})+\xi$ for all $y$ in $N_{x}$. But $f(l_{y})\le f(T^m(y))$ for all $y\in M$ so that $f(l_{y})\in S_{\xi}$ for all $y\in N_{x}$. Hence $S_{\xi}$ is open. Let $N(\xi)=S_{\xi}\cap N_{x_0}$. Since $x_0$ is an element of $S_{\xi}$ for all positive $\xi$, $N(\xi)$ is not empty for $\xi>0$. Since $N(\xi)$ is contained in $N_{x_0}$ and $S_{\xi}$, $f(l_{x_0})\le f(l_{x})\le f(l_{x_0})+\xi$ for all $x$ in $N(\xi)$. Since the points of $l_{x}$ are in the fixed point set $F$, $f(l_{x})$ can assume only finitely many values. Hence for $\xi$ sufficiently small 
\[
f(l_{x_0})\le f(l_{x})\le f(l_{x_0})+\xi
\]
implies that $f(l_{x})=f(l_{x_0})$, and so for some $\xi$, $x\in N(\xi)$ implies that $f(l_{x})=f(l_{x_0})$. Next, suppose $x_0$ is an element of $M$ and there is no neighborhood $U$ of $x_0$ in $M$ such that $f(l_{x})=f(l_{x_0})$ for all $x$ in $U$. Let $N$ be a neighborhood of $x_0$. According to the above arguments, there is an element $x$ of $N$ such that $f(l_{x})<f(l_{x_0})$. Let $K$ be the least upper bound of $f(L_{x})$ for $x$ in $N$. Since the range of $f(l_{x})$ is finite, there is a point $y$ of $N$ such that $f(l_{y})=K$. Thus $f(l_{y})\le f(l_{x})$ for all $x$ in $N$, and $N$ is a neighborhood of $y$. Thus there is a neighborhood $U$ of $y$ such that $f(l_{y})=f(l_{x})$ for all $x\in U$. Let $U_N=N\cap U$, the proof completes.
\end{proof}

Next we complete the proof of theorem. For each $x\in M-F$, $\Phi(x)$ is convergent. Let $G_1$ be the set of all elements $x$ of $M$ such that $f(L_{x})$ is constant in a neighborhood of $x$. Let $G_2$ be the set of all elements $x$ of $M$ such that $f(l_{x})$ is a constant in a neighborhood of $x$. Notice that $G=(M-F)\cap G_1\cap G_2$ is an open dense subset of $M-F$. For each $x\in M$, let 
\[
S(x)=\sum_{n=-\infty}^{\infty}f(T^{n-1}(x))-f(T^n(x)).
\]
Clearly $S(x)$ converges at each $x$ to $f(L_{x})-f(l_{x})$. Let $y$ be an element of $G$. There is a neighborhood $U$ of $y$ such that $S(x)$ represents the constant function in $U$. Since $y\notin F$ and $F$ is compact, there is a neighborhood $W$ of $y$ such that $\bar{W}\subset U\cap V$. Then $S(x)$ is a series of positive terms converging to a continuous function on $\bar{W}$ and $S(x)$ converges uniformly on $\bar{W}$. Let $p(x)$ be any bounded function continuous on $M$. The series
\[
\Phi(x)=\sum_{n=-\infty}^{\infty}p(T^n(x))\left(f(T^{n-1}(x))-f(T^n(x))\right)
\]
converges uniformly on $\bar{W}$ since $f$ is taken to be bounded on $M$. Since $p$, $f$ and $T$ are continuous , $\Phi(x)$ is continuous on $\bar{W}$ and hence at $y$. The invariance is obvious, so the proof completes.
\end{proof}

\end{document}